\documentclass[12pt,reqno]{amsart}
\usepackage{amsthm,amsfonts,amssymb,euscript}

\def\normo#1{\left\|#1\right\|}
\def\normb#1{\big\|#1\big\|}

\def\aabs#1{\left|#1\right|}

\def\rev#1{\frac{1}{#1}}
\def\half#1{\frac{#1}{2}}
\def\norm#1{\|#1\|}
\def\jb#1{\langle#1\rangle}
\def\wt#1{\widetilde{#1}}
\def\wh#1{\widehat{#1}}

\newcommand{\T}{{\mathbb T}}

\newcommand{\R}{{\mathbb R}}
\newcommand{\C}{{\mathbb C}}
\newcommand{\Z}{{\mathbb Z}}
\newcommand{\ft}{{\mathcal{F}}}

\newcommand{\les}{{\lesssim}}
\newcommand{\ges}{{\gtrsim}}
\newcommand{\ra}{{\rightarrow}}

\newcommand{\Sch}{{\mathcal{S}}}
\newcommand{\supp}{{\mbox{supp}}}

\theoremstyle{plain}
  \newtheorem{theorem}[subsection]{Theorem}
  \newtheorem{proposition}[subsection]{Proposition}
  \newtheorem{lemma}[subsection]{Lemma}
  \newtheorem{corollary}[subsection]{Corollary}

\theoremstyle{remark}
  \newtheorem{remark}[subsection]{Remark}

\theoremstyle{definition}

\numberwithin{equation}{section}

\begin{document}
\title[Dispersion Generalized Benjamin-Ono equation]{Local Well-posedness for dispersion generalized Benjamin-Ono equations in Sobolev spaces}
%\thanks{}
\author{Zihua Guo}
\address{LMAM, School of Mathematical Sciences, Peking University, Beijing
100871, China}

\email{guozihua@gmail.com}

\begin{abstract}
We prove that the Cauchy problem for the dispersion generalized
Benjamin-Ono equation
\[\partial_t u+|\partial_x|^{1+\alpha}\partial_x u+uu_x=0,\ u(x,0)=u_0(x),\]
is locally well-posed in the Sobolev spaces $H^s$ for $s>1-\alpha$
if $0\leq \alpha \leq 1$. The new ingredient is that we develop the
methods of Ionescu, Kenig and Tataru \cite{IKT} to approach the
problem in a less perturbative way, in spite of the ill-posedness
results of Molinet, Saut and Tzvetkovin \cite{MST}. Moreover, as a
bi-product we prove that if $0<\alpha \leq 1$ the corresponding
modified equation (with the nonlinearity $\pm uuu_x$) is locally
well-posed in $H^s$ for $s\geq 1/2-\alpha/4$.
\end{abstract}


\maketitle

%\tableofcontents

\section{Introduction}

In this paper, we consider the Cauchy problem for the dispersion
generalized Benjamin-Ono equation
\begin{eqnarray}\label{eq:dgBO}
\left\{
\begin{array}{l}
\partial_t u+|\partial_x|^{1+\alpha}\partial_x u+uu_x=0,\ (x,t)\in \R^2,\\
u(x,0)=u_0(x),
\end{array}
\right.
\end{eqnarray}
where $0\leq \alpha \leq 1$, $u: \R^2 \ra \R$ is a real-valued
function and $|\partial_x|$ is the Fourier multiplier operator with
symbol $|\xi|$. These equations arise as mathematical models for the
weakly nonlinear propagation of long waves in shallow channels. Note
that the case $\alpha=0$ corresponds to the Benjamin-Ono equation
and the case $\alpha=1$ corresponds to the Korteweg-de Vries
equation. During the past decades, both of the two equations were
extensively studied in a large number of literatures
\cite{TaoBO,IK,Bour,KPV,KPV2,I-method}. For example, see
\cite{Taopage} for a thorough review.

In proving the well-posedness of the Cauchy problem \eqref{eq:dgBO}
by direct contraction principle, the biggest enemy is the loss of
derivative from the nonlinearity. It was proved by Molinet, Saut and
Tzvetkov \cite{MST} that if $0\leq \alpha <1$ then $H^s$ assumption
alone on the initial data is insufficient for a proof of local
well-posedness of \eqref{eq:dgBO} via Picard iteration by showing
the solution mapping fails to be $C^2$ smooth from $H^s$ to
$C([0,T];H^s)$ at the origin for any $s$. It is due to that the
dispersive effect of the dispersive group of Eq. \eqref{eq:dgBO}
when $0\leq \alpha <1$ is too weak to spread the derivative in the
nonlinearity and hence the $high\times low$ interactions break down
the $C^2$ smoothness. When $\alpha=0$ a stronger ill-posedness was
proved by Koch and Tzvetkov \cite{KoTz} that the solution mapping
actually fails to be locally uniformly continuous in $H^s$ for any
$s$. For the positive side, some weaker well-posedness results (only
require the solution mapping to be continuous) were obtained. For
the Benjamin-Ono equation ($\alpha=0$),  Tao \cite{TaoBO} obtained
the global wellposedness in $H^s$ for $s\geq 1$ by performing a
gauge transformation as for the derivative Schr\"odinger equation.
This result was improved to $s\geq 0$ by Ionescu and Kenig
\cite{IK}. For the KdV equation ($\alpha=1$), the first
well-posedness by contraction principle was due to Kenig, Ponce and
Vega \cite{KPV2} who obtained LWP in $H^s$ for $s>3/4$. Bourgain
\cite{Bour} extended this result to GWP in $L^2$ by developing
$X^{s,b}$ space. Then Kenig, Ponce and Vega \cite{KPV} obtained
local well-posedness in $H^{s}$ for $s>-3/4$ and Colliander, Keel,
Staffilani, Takaoka and Tao \cite{I-method} extended it to a global
result where $I-method$ was introduced. Local well-posedness in
$H^{-3/4}$ was obtained by Christ, Colliander, and Tao \cite{CCT}
using miura transform and the $H^{1/4}$ local well-posedness for the
modified KdV equation. Recently, the author \cite{GuoKdV} obtained
global well-posedness in $H^{-3/4}$ by using directly the
contraction principle to prove local well-posedness.

This paper is mainly concerned with Eq. \eqref{eq:dgBO} for $0\leq
\alpha<1$. If $0< \alpha<1$, Kenig, Ponce and Vega \cite{KPV3} have
shown that \eqref{eq:dgBO} is locally well-posed for data in $H^s$
provided $s\geq \frac{3}{4}(2-\alpha)$ using the energy method
enhanced with the smoothing effect. In \cite{CKS} Colliander, Kenig
and Staffilani obtained LWP for the data lying in some weighted
Sobolev spaces by applying Picard iteration. S. Herr \cite{Herr,
Herr2} obtained LWP in $H^s\cap \dot{H}^{\rev{2}-\rev{1+\alpha}}$
for $s>-\frac{3}{4}\alpha$ and global well-posedness for $s\geq 0$
by requiring the initial data has additional properties in low
frequency to make the contraction principle work. Compared to the
Benjamin-Ono equation, the dispersive group of \eqref{eq:dgBO} has
stronger dispersive effect but it seems difficult to apply a gauge
transform to \eqref{eq:dgBO} which can weaken the high-low
interaction\footnote{Personal communication. Recently, the author
learned that Ionescu, Herr, Kenig and Koch proved $L^2$
well-posedness by using a paradifferential gauge.}. Moreover, when
$0<\alpha<1$, Eq. \eqref{eq:dgBO} is not completely integrable, but
there are still at least the following three conservation laws: if
$u$ is a smooth solution to \eqref{eq:dgBO} then
\begin{eqnarray}\label{eq:conservation}
&&\frac{d}{dt}\int_\R u(x,t)dx=0, \\
&&\frac{d}{dt}\int_\R u(x,t)^2dx=0,\label{eq:L2con}\\
&&\frac{d}{dt}\int_\R
||\partial_x|^{\frac{1+\alpha}{2}}u|^2+\frac{1}{6}
u(x,t)^3dx=0.\label{eq:H1half}
\end{eqnarray}
These conservation laws provide a-priori bounds on the solution. For
example, we can easily get from \eqref{eq:L2con} and
\eqref{eq:H1half} that for any smooth solution $u$ of
\eqref{eq:dgBO} on $[-T,T]$ then we have
\begin{eqnarray}\label{eq:apriori}
\norm{u(t)}_{H^{\frac{1+\alpha}{2}}}\leq
C(\norm{u_0}_{H^{\frac{1+\alpha}{2}}}),\quad \forall t\in [-T,T].
\end{eqnarray}
We will also need another symmetry. It is easy to see that Eq.
\eqref{eq:dgBO} is invariant under the following scaling transform
for any $\lambda>0$
\begin{equation}\label{eq:scaling}
u(x,t)\ra\ u_\lambda(x,t)=\lambda^{1+\alpha}u(\lambda
x,\lambda^{2+\alpha}t), \quad
u_{0,\lambda}=\lambda^{1+\alpha}u_0(\lambda x).
\end{equation}
Then we see $\dot{H}^{-\rev{2}-\alpha}$ is the critical space in the
sense of the scaling \eqref{eq:scaling}
\[\norm{u_{0,\lambda}}_{\dot{H}^{-\rev{2}-\alpha}}=\norm{u_0}_{\dot{H}^{-\rev{2}-\alpha}}.\]

Now we state our main results:

\begin{theorem}\label{thmmain}
(a) Let $0\leq \alpha \leq 1$. Assume $s>1-\alpha$ and $u_0\in
H^\infty$. Then there exists $T=T(\norm{u_0}_{H^s})>0$ such that
there is a unique solution $u=S^\infty_T(u_0)\in C([-T,T]:H^\infty)$
of the Cauchy problem \eqref{eq:dgBO}. In addition, for any
$\sigma\geq s$
\begin{eqnarray}
\sup_{|t|\leq T} \norm{S^\infty_T(u_0)(t)}_{H^\sigma}\leq
C(T,\sigma,\norm{u_0}_{H^\sigma}).
\end{eqnarray}

(b) Moreover, the mapping $S_T^\infty:H^\infty\rightarrow
C([-T,T]:H^\infty)$ extends uniquely to a continuous mapping
\[S_T^s:H^s\rightarrow C([-T,T]:H^s).\]
\end{theorem}

If $1/3<\alpha \leq 1$, then $1-\alpha<\half{1+\alpha}$. Thus from
the a-priori bound \eqref{eq:apriori}, and iterating Theorem
\ref{thmmain}, we obtain the following corollary.
\begin{corollary}
The Cauchy problem \eqref{eq:dgBO} is globally wellposed in $H^{s}$
for $s\geq\half{1+\alpha}$ if $1/3<\alpha\leq 1$.
\end{corollary}

It is easy to see $1-\alpha < \frac{3}{4}(2-\alpha)$ for $0\leq
\alpha \leq 1$, thus our results improve the results in \cite{KPV3}.
We discuss now some of the ingredients in the proof of Theorem
\ref{thmmain}. We will adapt the ideas of Ionescu, Kenig and Tataru
\cite{IKT} to approach the problem in a less perturbative way. It
can be viewed as a combination of the energy methods and the
perturbative methods. More precisely, we will define $F^{s}(T),
N^{s}(T)$ and energy space $E^{s}(T)$ and show that if $u$ is a
smooth solution of \eqref{eq:dgBO} on $\R\times [-T,T]$ then
\begin{eqnarray}\label{eq:scheme}
\left \{
\begin{array}{l}
\norm{u}_{F^{s}(T)}\les \norm{u}_{E^{s}(T)}+\norm{\partial_x(u^2)}_{N^{s}(T)};\\
\norm{\partial_x(u^2)}_{N^{s}(T)}\les \norm{u}_{F^{s}(T)}^2;\\
\norm{u}_{E^{s}(T)}^2\les
\norm{\phi}_{{H}^s}^2+\norm{u}_{F^{s}(T)}^3.
\end{array}
\right.
\end{eqnarray}
The inequalities \eqref{eq:scheme} and a continuity argument still
suffice to control $\norm{u}_{F^{s}(T)}$, provided that
$\norm{\phi}_{{H}^s}\ll 1$ (which can by arranged by rescaling if
$s\geq 0$). The first inequality in \eqref{eq:scheme} is the
analogue of the linear estimate. The second inequality in
\eqref{eq:scheme} is the analogue of the bilinear estimate. The last
inequality in \eqref{eq:scheme} is an energy-type estimate. To prove
Theorem \ref{thmmain} (b), we need to study the difference equation
of Eq. \eqref{eq:dgBO}. This difference equation has less
symmetries, but some special symmetries for real-valued solutions in
$L^2$. We then follow the methods in \cite{IKT} to prove the
continuity of the solution mapping in $H^s$ by adapting the
Bona-Smith method \cite{BonaSmith}.

We will develop the ideas in \cite{IKT} to define the main normed
and semi-normed spaces. As was explained before, standard using of
$X^{s,b}$ spaces in fixed-point argument does not work for
\eqref{eq:dgBO}. But we use $X^{s,b}$-type structures only on small,
frequency dependant time intervals. The high-low interaction can be
controlled for short time. The length of the time interval will be
important. Generally, one need to control the interaction in as
large time interval as possible and leave the rest to be controlled
in the energy estimates. We will choose the length which will just
suffice to control the high-low interaction. Since we only control
the interaction in short time then we need to define
$\norm{u}_{E^{s}(T)}$ sufficiently large to be able to still prove
the linear estimate in \eqref{eq:scheme}. Finally, we use
frequence-localized energy estimates and the symmetries of the
equation \eqref{eq:dgBO} to prove the energy estimates.

As a bi-product, we use our estimates for the multiplier to study
the following modified equation
\begin{eqnarray}\label{eq:mdgBO}
\partial_t u+|\partial_x|^{1+\alpha}\partial_x u\mp u^2u_x=0, \quad
u(x,0)=u_0(x).
\end{eqnarray}
When $\alpha=0$ and $\alpha=1$, it corresponds to the modified
Benjamin-Ono equation and modified Korteweg-de Vries equation. Both
equations were also extensively studied \cite{MR,MR2,KenigT,Guo}.
The high-low interactions in the trilinear estimates are much weaker
than that in the bilinear estimates. Indeed, it is known that for
$\alpha=1$ the high-low interactions are under control \cite{KPV}
and for $\alpha=0$ the high-low interactions only cause logarithmic
divergence which is removable \cite{Guo}. So it is natural to
conjecture that for $\alpha>0$ the high-low interactions are also
under control and a direct using of $X^{s,b}$ space would suffice
for a well-posedness as in \cite{KPV}. We proved the following

\begin{theorem}\label{thmmdg}
Let $0<\alpha \leq 1$ and $\phi \in H^s$ for $s\geq 1/2-\alpha/4$.
Then there exist $T=T(\norm{\phi}_{H^{1/2-\alpha/4}})>0$ and a
unique solution $u\in X_T^{s,1/2+}$ to \eqref{eq:mdgBO} on $(-T,T)$.
Moreover, the solution mapping $\phi \rightarrow u$ is locally
Lipschitz continuous from $H^s$ to $C([-T,T]:H^s)$.
\end{theorem}

On the other hand, the equation \eqref{eq:mdgBO} has also several
conservation laws: if $u$ is a smooth solution to \eqref{eq:mdgBO}
then
\begin{eqnarray}\label{eq:mdgBOconservation}
&&\frac{d}{dt}\int_\R u(x,t)dx=0, \\
&&\frac{d}{dt}\int_\R u(x,t)^2dx=0,\label{eq:L2conmdg}\\
&&\frac{d}{dt}\int_\R
||\partial_x|^{\frac{1+\alpha}{2}}u|^2\pm\frac{1}{12}
u(x,t)^4dx=0.\label{eq:H1halfmdg}
\end{eqnarray}
It is easy to see that Eq. \eqref{eq:mdgBO} is invariant under the
following scaling transform: for any $\lambda>0$
\begin{equation}\label{eq:scalingmdg}
u(x,t)\ra\ u_\lambda(x,t)=\lambda^{\half{1+\alpha}}u(\lambda
x,\lambda^{2+\alpha}t), \quad
u_{0,\lambda}=\lambda^{\half{1+\alpha}}u_0(\lambda x).
\end{equation}
Then we see $L^2$ is subcritical space in the sense of the scaling
and easily obtain the a-priori bound: if $u$ is a smooth solution to
\eqref{eq:mdgBO} (both foucusing and defocusing) then for any $t\in
\R$
\begin{eqnarray}\label{eq:apriorimdg}
\norm{u}_{H^{\half{1+\alpha}}}\leq
C(\norm{\phi}_{H^{\half{1+\alpha}}}).
\end{eqnarray}

From $1/2-\alpha/4\leq \half{1+\alpha}$ and the a-priori bound
\eqref{eq:apriorimdg}, and iterating Theorem \ref{thmmain}, we
obtain the following corollary.
\begin{corollary}
The Cauchy problem \eqref{eq:mdgBO} is globally wellposed in $H^{s}$
for $s\geq\half{1+\alpha}$ if $0<\alpha\leq 1$.
\end{corollary}

The rest of the paper is organized as follows: In section 2 we
present some notations and Banach function spaces. The estimates for
the characterization multiplier will be given in section 3. In
section 4 we prove Theorem \ref{thmmdg}. In section 5 we prove some
short-time bilinear estimates. We prove Theorem \ref{thmmain} in
section 6 using the energy estimates obtained in section 7.

\section{Notation and Definitions}

Throughout this paper, we fix $0\leq \alpha <1$. For $x, y\in \R^+$,
$x\les y$ means that there exists $C>0$ such that $x\leq Cy$. By
$x\sim y$ we mean $x\les y$ and $y\les x$. For $f\in \Sch'$ we
denote by $\widehat{f}$ or $\ft (f)$ the Fourier transform of $f$
for both spatial and time variables,
\begin{eqnarray*}
\widehat{f}(\xi, \tau)=\int_{\R^2}e^{-ix \xi}e^{-it \tau}f(x,t)dxdt.
\end{eqnarray*}
Moreover, we use $\ft_x$ and $\ft_t$ to denote the Fourier transform
with respect to space and time variable respectively. Let
$\Z_+=\Z\cap[0, \infty)$. Let $I_{\leq 0}=\{\xi: |\xi|<3/2\}$,
$\wt{I}_{\leq 0}=\{\xi:|\xi|\leq 2\}$. For $k\in \Z$ let
\[I_k=\{\xi:|\xi|\in [(3/4)\cdot 2^k,
(3/2)\cdot 2^k)\ \}\mbox{ and }\widetilde{I}_k=\{\xi: |\xi|\in
[2^{k-1}, 2^{k+1}]\ \}.\]

Let $\eta_0: \R\rightarrow [0, 1]$ denote an even smooth function
supported in $[-8/5, 8/5]$ and equal to $1$ in $[-5/4, 5/4]$. For
$k\in \Z$ let $\chi_k(\xi)=\eta_0(\xi/2^k)-\eta_0(\xi/2^{k-1})$,
$\chi_k$ supported in $\{\xi: |\xi|\in[(5/8)\cdot 2^k, (8/5)\cdot
2^k]\}$, and
\[\chi_{[k_1,k_2]}=\sum_{k=k_1}^{k_2}\chi_k \mbox{ for any } k_1\leq k_2\in \Z.\]
For simplicity, let $\eta_k=\chi_k$ if $k\geq 1$ and $\eta_k\equiv
0$ if $k\leq -1$. Also, for $k_1\leq k_2\in \Z$ let
\[\eta_{[k_1,k_2]}=\sum_{k=k_1}^{k_2}\eta_k \mbox{ and }\eta_{\leq k_2}=\sum_{k=-\infty}^{k_2}\eta_{k}.\]
Roughly speaking, $\{\chi_k\}_{k\in \mathbb{Z}}$ is the homogeneous
decomposition function sequence and $\{\eta_k\}_{k\in \mathbb{Z}_+}$
is the non-homogeneous decomposition function sequence to the
frequency space. For $k\in \Z$ let $P_k$ denote the operators on
$L^2(\R)$ defined by
$\widehat{P_ku}(\xi)=1_{I_k}(\xi)\widehat{u}(\xi)$. By a slight
abuse of notation we also define the operators $P_k$ on
$L^2(\R\times \R)$ by formulas $\ft(P_ku)(\xi,
\tau)=1_{I_k}(\xi)\ft(u)(\xi, \tau)$. For $l\in \Z$ let
\[
P_{\leq l}=\sum_{k\leq l}P_k, \quad P_{\geq l}=\sum_{k\geq l}P_k.
\]

For $x\in \R$, let $[x]$ be the largest integer that is less or
equal to $x$. Let $a_1, a_2, a_3\in \R$. It will be convenient to
define the quantities $a_{max}\geq a_{med}\geq a_{min}$ to be the
maximum, median, and minimum of $a_1,a_2,a_3$ respectively. Usually
we use $k_1,k_2,k_3$ and $j_1,j_2,j_3$ to denote integers,
$N_i=2^{k_i}$ and $L_i=2^{j_i}$ for $i=1,2,3$ to denote dyadic
numbers.

For $\xi\in \R$ let
\begin{equation}\label{eq:dr}
\omega(\xi)=-\xi|\xi|^{1+\alpha}
\end{equation}
which is the dispersion relation associated to Eq. \eqref{eq:dgBO}.
For $\phi \in L^2(\R)$ let $W(t)\phi\in C(\R: L^2)$ denote the
solution of the free evolution given by
\begin{equation}
\ft_x[W(t)\phi](\xi,t)=e^{it\omega(\xi)}\widehat{\phi}(\xi).
\end{equation}
We introduce the $X^{s,b}$ norm associated to the equation
\eqref{eq:dgBO} which is given by
\begin{eqnarray*}
\norm{u}_{X^{s,b}}=\norm{\jb{\tau-\omega(\xi)}^b\jb{\xi}^s\widehat{u}(\xi,\tau)}_{L^2(\R^2)},
\end{eqnarray*}
where $\jb{\cdot}=(1+|\cdot|^2)^{1/2}$. The spaces $X^{s,b}$ turn
out to be very useful in the study of low-regularity theory for the
dispersive equations. These spaces were first used to systematically
study nonlinear dispersive wave problems by Bourgain \cite{Bour} and
developed by Kenig, Ponce and Vega \cite{KPV} and Tao \cite{Taokz}.
Klainerman and Machedon \cite{KlMa} used similar ideas in their
study of the nonlinear wave equation. We denote by $X^{s,b}_T$ the
space that $X^{s,b}$ localized to the interval $[-T,T]$.

For $k,j \in \Z_+$ let \[D_{k,j}=\{(\xi, \tau)\in \R \times \R: \xi
\in \widetilde{I}_k, \tau-\omega(\xi)\in \widetilde{I}_j\},\quad
D_{k,\leq j}=\cup_{l\leq j}D_{k,l}.\] For $k\in \Z_+$ we define the
dyadic $X^{s,b}$-type normed spaces $X_k(\R^2)$:
\begin{eqnarray}\label{eq:Xk}
X_k&=&\{f\in L^2(\R^2): f(\xi,\tau) \mbox{ is supported in }
\widetilde
{I}_k\times\R  \mbox{ ($\wt{I}_{\leq 0}\times \R$ if $k=0$)} \nonumber\\
&&\mbox{ and }\norm{f}_{X_k}:=\sum_{j=0}^\infty
2^{j/2}\norm{\eta_j(\tau-w(\xi))\cdot
f(\xi,\tau)}_{L^2_{\xi,\tau}}<\infty\}.
\end{eqnarray}
These $l^1$-type $X^{s,b}$ structures were first introduced in
\cite{Tataru} and used in \cite{IK, IKT, Taoscatteringgkdv, GW}.

The definition shows easily that if $k\in \Z_+$ and $f_k\in X_k$
then
\begin{eqnarray}\label{eq:pXk1}
\normo{\int_{\R}|f_k(\xi,\tau')|d\tau'}_{L_\xi^2}\les
\norm{f_k}_{X_k}.
\end{eqnarray}
Moreover, it is easy to see (see \cite{Guo}) that if $k\in \Z_+$,
$l\in \Z_+$, and $f_k\in X_k$ then
\begin{eqnarray}\label{eq:pXk2}
&&\sum_{j=l+1}^\infty 2^{j/2}\normo{\eta_j(\tau-\omega(\xi))\cdot
\int_{\R}|f_k(\xi,\tau')|\cdot
2^{-l}(1+2^{-l}|\tau-\tau'|)^{-4}d\tau'}_{L^2}\nonumber\\
&&+2^{l/2}\normo{\eta_{\leq l}(\tau-\omega(\xi)) \int_{\R}
|f_k(\xi,\tau')| 2^{-l}(1+2^{-l}|\tau-\tau'|)^{-4}d\tau'}_{L^2}\les
\norm{f_k}_{X_k}.
\end{eqnarray}
In particular, if $k\in \Z_+$, $l\in \Z_+$, $t_0\in \R$, $f_k \in
X_k$, and $\gamma\in \Sch(\R)$, then
\begin{eqnarray}\label{eq:pXk3}
\norm{\ft[\gamma(2^l(t-t_0))\cdot \ft^{-1}(f_k)]}_{X_k}\les
\norm{f_k}_{X_k}.
\end{eqnarray}

As in \cite{IKT} at frequency $2^k$ we will use the $X^{s, b}$
structure given by the $X_k$ norm, uniformly on the
$2^{-[(1-\alpha)k]}$ time scale. We will explain briefly why we use
this scale in the next section. For $k\in \Z_+$ we define the normed
spaces
\begin{eqnarray*}
&& F_k=\left\{
\begin{array}{l}
f\in L^2(\R^2): \widehat{f}(\xi,\tau) \mbox{ is supported in }
\widetilde{I}_k\times\R \mbox{ ($\wt{I}_{\leq 0}\times \R$ if $k=0$)} \\
\mbox{ and }\norm{f}_{F_k}=\sup\limits_{t_k\in \R}\norm{\ft[f\cdot
\eta_0(2^{[(1-\alpha)k]}(t-t_k))]}_{X_k}<\infty
\end{array}
\right\},
\\
&&N_k=\left\{
\begin{array}{l}
f\in L^2(\R^2): \supp \widehat{f}(\xi,\tau) \subset
\widetilde{I}_k\times\R \mbox{ ($\wt{I}_{\leq 0}\times \R$ if $k=0$)} \mbox{ and } \norm{f}_{N_k}=\\
\sup\limits_{t_k\in
\R}\norm{(\tau-\omega(\xi)+i2^{[(1-\alpha)k]})^{-1}\ft[f\cdot
\eta_0(2^{[(1-\alpha)k]}(t-t_k))]}_{X_k}<\infty
\end{array}
\right\}.
\end{eqnarray*}
We see from the definitions that we still use $X^{s,b}$ structure on
the whole interval for the low frequency. Since the spaces $F_k$ and
$N_k$ are defined on the whole line, we define then local versions
of the spaces in standard ways. For $T\in (0,1]$ we define the
normed spaces
\begin{eqnarray*}
F_k(T)&=&\{f\in C([-T,T]:L^2): \norm{f}_{F_k(T)}=\inf_{\wt{f}=f
\mbox{ in } \R\times [-T,T]}\norm{\wt f}_{F_k}\};\\
N_k(T)&=&\{f\in C([-T,T]:L^2): \norm{f}_{N_k(T)}=\inf_{\wt{f}=f
\mbox{ in } \R\times [-T,T]}\norm{\wt f}_{N_k}\}.
\end{eqnarray*}
We assemble these dyadic spaces in a Littlewood-Paley manner. For
$s\geq 0$ and $T\in (0,1]$, we define the normed spaces
\begin{eqnarray*}
&&F^{s}(T)=\left\{ u:\
\norm{u}_{F^{s}(T)}^2=\sum_{k=1}^{\infty}2^{2sk}\norm{P_k(u)}_{F_k(T)}^2+\norm{P_{\leq
0}(u)}_{F_0(T)}^2<\infty \right\},
\\
&&N^{s}(T)=\left\{ u:\
\norm{u}_{N^{s}(T)}^2=\sum_{k=1}^{\infty}2^{2sk}\norm{P_k(u)}_{N_k(T)}^2+\norm{P_{\leq
0}(u)}_{N_0(T)}^2<\infty \right\}.
\end{eqnarray*}
We define the dyadic energy space. For $s\geq 0$ and $u\in
C([-T,T]:H^\infty)$ we define
\begin{eqnarray*}
\norm{u}_{E^{s}(T)}^2=\norm{P_{\leq 0}(u(0))}_{L^2}^2+\sum_{k\geq
1}\sup_{t_k\in [-T,T]}2^{2sk}\norm{P_k(u(t_k))}_{L^2}^2.
\end{eqnarray*}

As in \cite{IKT}, for any $k\in \Z_+$ we define the set $S_k$ of
$k-acceptable$ time multiplication factors
\begin{eqnarray}
S_k=\{m_k:\R\rightarrow \R: \norm{m_k}_{S_k}=\sum_{j=0}^{10}
2^{-j[(1-\alpha)k]}\norm{\partial^jm_k}_{L^\infty}< \infty\}.
\end{eqnarray}
For instance, $\eta(2^{[(1-\alpha)k]}t) \in S_k$ for any $\eta$
satisfies $\norm{\partial_x^j \eta}_{L^\infty}\leq C$ for
$j=0,1,2,\ldots, 10$. Direct estimates using the definitions and
\eqref{eq:pXk2} show that for any $s\geq 0$ and $T\in (0,1]$
\begin{eqnarray}\label{eq:Sk}
\left \{
\begin{array}{l}
\norm{\sum_{k\in \Z_+} m_k(t)\cdot P_k(u)}_{F^{s}(T)}\les (\sup_{k\in \Z_+}\norm{m_k}_{S_k})\cdot \norm{u}_{F^{s}(T)};\\
\norm{\sum_{k\in \Z_+} m_k(t)\cdot P_k(u)}_{N^{s}(T)}\les
(\sup_{k\in \Z_+}\norm{m_k}_{S_k})\cdot \norm{u}_{N^{s}(T)};\\
\norm{\sum_{k\in \Z_+} m_k(t)\cdot P_k(u)}_{E^{s}(T)}\les
(\sup_{k\in \Z_+}\norm{m_k}_{S_k})\cdot \norm{u}_{E^{s}(T)}.
\end{array}
\right.
\end{eqnarray}

\section{A Symmetric Estimate}

In this section we prove symmetric estimates which will be used to
prove bilinear estimates. For $\xi_1,\xi_2 \in \R$ and $\omega:\R
\rightarrow \R$ as in \eqref{eq:dr} let
\begin{equation}\label{eq:reso}
\Omega(\xi_1,\xi_2)=\omega(\xi_1)+\omega(\xi_2)-\omega(\xi_1+\xi_2).
\end{equation}
This is the resonance function that plays a crucial role in the
bilinear estimate of the $X^{s,b}$-type space. See \cite{Taokz} for
a perspective discussion. For compactly supported nonnegative
functions $f,g,h\in L^2(\R\times \R)$ let
\begin{eqnarray*}
J(f,g,h)=\int_{\R^4}f(\xi_1,\mu_1)g(\xi_2,\mu_2)h(\xi_1+\xi_2,\mu_1+\mu_2+\Omega(\xi_1,\xi_2))d\xi_1d\xi_2d\mu_1d\mu_2.
\end{eqnarray*}

\begin{lemma}\label{lemsymes}
Assume $k_i \in \Z$, $j_i\in \Z_+$,
$N_i=2^{k_i},L_i=2^{j_i},i=1,2,3$. Let $f_{k_i,j_i}\in L^2(\R\times
\R)$ are nonnegative functions supported in
$[2^{k_i-1},2^{k_i+1}]\times \widetilde{I}_{j_i}, \ i=1,\ 2,\ 3$.
Then

(a) For any $k_1, k_2, k_3\in \Z$ and $j_1,j_2,j_3\in \Z_+$,
\begin{equation}
J(f_{k_1,j_1},f_{k_2,j_2},f_{k_3,j_3})\leq C
2^{j_{min}/2}2^{k_{min}/2} \prod_{i=1}^3\norm{f_{k_i,j_i}}_{L^2}.
\end{equation}

(b) If $N_{min}\ll N_{med}\sim N_{max}$ and $(k_i,j_i)\neq
(k_{min},j_{max})$ for all $i=1,2,3$,
\begin{equation}\label{eq:lemsymesrb}
J(f_{k_1,j_1},f_{k_2,j_2},f_{k_3,j_3})\leq C
2^{(j_{min}+j_{med})/2}2^{-(1+\alpha)k_{max}/2}\prod_{i=1}^3\norm{f_{k_i,j_i}}_{L^2};
\end{equation}
If $N_{min}\ll N_{med}\sim N_{max}$ and $(k_i,j_i)=
(k_{min},j_{max})$ for some $i\in \{1,2,3\}$,
\begin{equation}
J(f_{k_1,j_1},f_{k_2,j_2},f_{k_3,j_3})\leq C
2^{(j_{min}+j_{med})/2}2^{-\alpha k_{max}/2}2^{-k_{min}/2}
\prod_{i=1}^3\norm{f_{k_i,j_i}}_{L^2}.
\end{equation}

(c) For any $k_1,k_2,k_3\in \Z$ with $N_{min}\sim N_{med}\sim
N_{max}\gg 1$ and $j_1,j_2,j_3\in \Z_+$
\begin{eqnarray}
J(f_{k_1,j_1},f_{k_2,j_2},f_{k_3,j_3})\leq C
2^{j_{min}/2}2^{j_{med}/4}2^{-\alpha k_{max}/4}
\prod_{i=1}^2\norm{f_{k_i,j_i}}_{L^2}.
\end{eqnarray}
\end{lemma}

\begin{proof}
Simple changes of variables in the integration and the observation
that the function $\omega$ is odd show that
\begin{eqnarray*}
|J(f,g,h)|=|J(g,f,h)|=|J(f,h,g)|=|J(\widetilde{f},g,h)|,
\end{eqnarray*}
where $\widetilde{f}(\xi,\mu)=f(-\xi,-\mu)$. Let
$A_{k_i}(\xi)=[\int_\R |f_{k_i,j_i}(\xi,\mu)|^2d\mu]^{1/2}$,
$i=1,2,3$. Using the Cauchy-Schwartz inequality and the support
properties of the functions $f_{k_i,j_i}$, we obtain
\begin{eqnarray*}
J(f_{k_1,j_1},f_{k_2,j_2},f_{k_3,j_3})&\les& 2^{j_{min}/2}\int_{\R^2}A_{k_1}(\xi_1)A_{k_2}(\xi_2)A_{k_3}(\xi_1+\xi_2)d\xi_1d\xi_2\\
&\les&
2^{k_{min}/2}2^{j_{min}/2}\prod_{i=1}^3\norm{f_{k_i,j_i}}_{L^2},
\end{eqnarray*}
which is part (a), as desired.

For part (b), in view of the support properties of the functions, it
is easy to see that $J(f_{k_1,j_1},f_{k_2,j_2},f_{k_3,j_3})\equiv 0$
unless
\begin{equation}\label{eq:freeq}
N_{max}\sim N_{med}.
\end{equation}
From symmetry we may assume $k_1\leq k_2\leq k_3$ and we have three
cases. If $j_3=j_{max}$ then we will prove that if
$g_i:\R\rightarrow \R_+$ are $L^2$ functions supported in
$\widetilde{I}_{k_i}$, $i=1,2$, and $g: \R^2\rightarrow \R_+$ is an
$L^2$ function supported in $\widetilde{I}_{k_3}\times
\widetilde{I}_{j_4}$, then
\begin{eqnarray}\label{eq:lemsymesb}
\int_{\R^3}g_1(\xi_1)g_2(\xi_2)g(\xi_1+\xi_2,\Omega(\xi_1,\xi_2))d\xi_1d\xi_2\les
2^{-(1+\alpha)k_{max}/2}\norm{g_1}_{L^2}\norm{g_2}_{L^2}\norm{g}_{L^2}.
\end{eqnarray}
This suffices for \eqref{eq:lemsymesrb} by Cauchy-Schwarz
inequality.

To prove \eqref{eq:lemsymesb}, we observe that since $N_1\ll N_2$
then $|\xi_1+\xi_2|\sim |\xi_2|$. By change of variable
$\xi'_1=\xi_1$, $\xi'_2=\xi_1+\xi_2$, we get that the left-hand side
of \eqref{eq:lemsymesb} is bounded by
\begin{eqnarray}\label{eq:lemsymesb2}
\int_{|\xi_1|\sim N_1,|\xi_2|\sim
N_2}g_1(\xi_1)g_2(\xi_2-\xi_1)g(\xi_2,\Omega(\xi_1,\xi_2-\xi_1))d\xi_1d\xi_2.
\end{eqnarray}
Note that in the integration area we have
\begin{eqnarray*}
\big|\frac{\partial}{\partial_{\xi_1}}\left[\Omega(\xi_1,\xi_2-\xi_1)\right]\big|=|\omega'(\xi_1)-\omega'(\xi_2-\xi_1)|\sim
N_2^{1+\alpha},
\end{eqnarray*}
where we use the fact $|\omega'(\xi)|=(2+\alpha)|\xi|^{1+\alpha}$
and $N_1\ll N_2$. By change of variable
$\mu_2=\Omega(\xi_1,\xi_2-\xi_1)$ we get that \eqref{eq:lemsymesb2}
is bounded by
\begin{eqnarray}
2^{-(1+\alpha)k_{max}/2}\norm{g_1}_{L^2}\norm{g_2}_{L^2}\norm{g}_{L^2}.
\end{eqnarray}

If $j_2=j_{max}$ then this case is identical to the case
$j_3=j_{max}$ in view of \eqref{eq:freeq}. If $j_1=j_{max}$ it
suffices to prove that if $g_i:\R\rightarrow \R_+$ are $L^2$
functions supported in $I_{k_i}$, $i=2,3$, and $g: \R^2\rightarrow
\R_+$ is an $L^2$ function supported in $\widetilde{I}_{k_1}\times
\widetilde{I}_{j_1}$, then
\begin{eqnarray}\label{eq:lemsymesb3}
\int_{\R^2}g_2(\xi_2)g_3(\xi_3)g(\xi_2+\xi_3,\Omega(\xi_2,\xi_3))d\xi_2d\xi_3\les
2^{-(\alpha
k_{max}-k_{min})/2}\norm{g_2}_{L^2}\norm{g_3}_{L^2}\norm{g}_{L^2}.
\end{eqnarray}
Indeed, by change of variables $\xi'_2=\xi_2,\xi'_3=\xi_2+\xi_3$ and
noting that in the integration area $|\xi'_2|\sim
2^{k_2},|\xi'_3|\sim 2^{k_1}$,
\begin{eqnarray*}
\big|\frac{\partial}{\partial_{\xi'_2}}\left[\Omega(\xi'_2,\xi'_3-\xi'_2)\right]\big|=|\omega'(\xi'_2)-\omega'(\xi'_3-\xi'_2)|\sim
N_2^{\alpha}N_{1},
\end{eqnarray*}
we get from Cauchy-Schwarz inequality that
\begin{eqnarray}
&&\int_{\R^2}g_2(\xi_2)g_3(\xi_3)g(\xi_2+\xi_3,\Omega(\xi_2,\xi_3))d\xi_2d\xi_3\nonumber\\
&&\les \int_{|\xi'_2|\sim 2^{k_2},|\xi'_3|\sim 2^{k_1}}g_2(\xi'_2)g_3(\xi'_3-\xi'_2)g(\xi'_3,\Omega(\xi'_2,\xi'_3-\xi'_2))d\xi'_2d\xi'_3\nonumber\\
&&\les2^{-\alpha
k_{max}/2}2^{-k_{min}/2}\norm{g_2}_{L^2}\norm{g_3}_{L^2}\norm{g}_{L^2},
\end{eqnarray}
which is \eqref{eq:lemsymesb3} as desired.

We prove now part (c). For simplicity of notations we write
$f_i=f_{k_i,j_i},i=1,2,3$. From symmetries we may assume
$j_3=j_{max}$ and $|\xi_1|\leq |\xi_2|$. Then writing
$d\sigma=d\xi_1d\xi_2d\mu_1d\mu_2$ we have
\begin{eqnarray}
&&\int_{\R^4}f_1(\xi_1,\mu_1)f_2(\xi_2,\mu_2)f_3(\xi_1+\xi_2,\mu_1+\mu_2+\Omega(\xi_1,\xi_2))d\xi_1d\xi_2d\mu_1d\mu_2
\nonumber\\
&=&\big(\int_{\xi_1\cdot \xi_2<0}+\int_{\xi_1\cdot
\xi_2>0}\big)f_1(\xi_1,\mu_1)f_2(\xi_2,\mu_2)
f_3(\xi_1+\xi_2,\mu_1+\mu_2+\Omega(\xi_1,\xi_2))d\sigma\nonumber\\
&:=&I+II.
\end{eqnarray}
For the contributions of $I$, noting that if $\xi_1\cdot \xi_2<0$
and $|\xi_1|\leq |\xi_2|$ then
\begin{eqnarray*}
\big|\frac{\partial}{\partial_{\xi_1}}\left[\Omega(\xi_1,\xi_2-\xi_1)\right]\big|=|\omega'(\xi_1)-\omega'(\xi_2-\xi_1)|\sim
N_2^{1+\alpha},
\end{eqnarray*}
thus by change of variable $\xi_2'=\xi_2-\xi_1$ and we get as for
the first inequality in part (b) that
\begin{eqnarray}
I\les
2^{(j_{min}+j_{med})/2}2^{-(1+\alpha)k_{max}/2}\prod_{i=1}^3\norm{f_{k_i,j_i}}_{L^2}.
\end{eqnarray}
Interpolating with part (a), we immediately get the bound
\eqref{eq:lemsymesc} for this term.

For the contribution of $II$, we break it into two parts
\begin{eqnarray}
II&=& \big(\int_{|\xi_1-\xi_2|\leq R}+\int_{|\xi_1-\xi_2|\geq
R}\big)f_1(\xi_1,\mu_1)f_2(\xi_2,\mu_2)\\
&&\quad \cdot
f_3(\xi_1+\xi_2,\mu_1+\mu_2+\Omega(\xi_1,\xi_2))d\sigma\nonumber\\
&:=&II_1+II_2,
\end{eqnarray}
where $R$ will be determined later. For $II_1$, we may assume
$j_1=j_{min}$ and as in part (a) we get
\begin{eqnarray}
II_1&=&\int_{|\xi_1|\leq R}f_1(\xi_1+\xi_2,\mu_1)f_2(\xi_2,\mu_2)
f_3(\xi_1+2\xi_2,\mu_1+\mu_2+\Omega(\xi_1+\xi_2,\xi_2))d\sigma
\nonumber \\
&\les& 2^{j_{min}/2}R^{1/2} \prod_{i=1}^3\norm{f_i}_{L^2}.
\end{eqnarray}
For $II_2$ we will use the bilinear Strichartz estimate (see Lemma
3.4, \cite{Herr}): for any $u_1,u_2\in \Sch$ then
\begin{eqnarray}\label{eq:bilinearstri}
\normo{\int_{\xi=\xi_1+\xi_2}\big||\xi_1|^{1+\alpha}-|\xi_2|^{1+\alpha}\big|_1^{1/2}
e^{it\omega(\xi_1)}\wh{u_1}(\xi_1)e^{it\omega(\xi_2)}\wh{u_1}(\xi_2)}_{L^2_{\xi,t}}\les
\norm{u_1}_{L^2}\norm{u_2}_{L^2},
\end{eqnarray}
where $|x|_1=|x|\cdot 1_{\geq 1}(|x|)$. In the integral area of
$II_2$, we have
\[||\xi_1|^{1+\alpha}-|\xi_2|^{1+\alpha}|^{1/2}\ges N_2^{\alpha/2}|\xi_1-\xi_2|^{1/2}\ges N_2^{\alpha/2}R^{1/2}.\]
We will choose $R$ such that $N_2^{\alpha/2}R^{1/2}\ges 1$. Thus we
get
\begin{eqnarray*}
II_2&\les&
N_2^{-\alpha/2}R^{-1/2}\int_{\R^4}\big||\xi_1|^{1+\alpha}-|\xi_2|^{1+\alpha}\big|_1^{1/2}f_1(\xi_1,\mu_1)f_2(\xi_2,\mu_2)\\
&&\cdot f_3(\xi_1+\xi_2,\mu_1+\mu_2+\Omega(\xi_1,\xi_2))d\sigma \\
&\les& N_2^{-\alpha/2}R^{-1/2}\int_{\R^4}\big||\xi_1|^{1+\alpha}-|\xi_2|^{1+\alpha}\big|_1^{1/2}f_1(\xi_1,\mu_1-\omega(\xi_1))\\
&&\cdot
f_2(\xi_2,\mu_2-\omega(\xi_2))f_3(\xi_1+\xi_2,\mu_1+\mu_2-\omega(\xi_1+\xi_2))d\sigma.
\end{eqnarray*}
Using Cauchy-Schwartz inequality and Plancherel's equality we get
that $II_2$ is dominated by
\begin{eqnarray*}
&&N_2^{-\alpha/2}R^{-1/2}\norm{f_3}_{L^2}\\
&&\cdot  \normo{\int_{\xi=\xi_1+\xi_2,\mu=\mu_1+\mu_2}\big||\xi_1|^{1+\alpha}-|\xi_2|^{1+\alpha}\big|_1^{1/2}f_1(\xi_1,\mu_1-\omega(\xi_1))f_2(\xi_2,\mu_2-\omega(\xi_2))}_{L_{\xi,\mu}^2}\\
&&\les \quad
N_2^{-\alpha/2}R^{-1/2}\norm{f_3}_{L^2}\\
&&\quad \norm{\int
\eta_{j_1}(\mu_1)\eta_{j_2}(\mu_2)e^{it(\mu_1+\mu_2)}\cdot\int_{\xi=\xi_1+\xi_2}\big||\xi_1|^{1+\alpha}-|\xi_2|^{1+\alpha}\big|_1^{1/2}\\
&&\quad \quad \cdot
e^{it\omega(\xi_1)}f_1(\xi_1,\mu_1)e^{it\omega(\xi_2)}f_2(\xi_2,\mu_2)d\mu_1d\mu_2}_{L_{\xi,t}^2}\\
&&\les N_2^{-\alpha/2}R^{-1/2} 2^{j_{min}/2}2^{j_{med}/2}
\prod_{i=1}^3\norm{f_i}_{L^2},
\end{eqnarray*}
where we used \eqref{eq:bilinearstri} in the last inequality.
Therefore, taking $R=2^{-\alpha k_{max}/2}2^{j_{med}/2}$ we complete
the proof of part (c).
\end{proof}

We restate now Lemma \ref{lemsymes} in a form that is suitable for
the trilinear estimates in the next sections.
\begin{corollary}\label{cor42}
Assume $k_1,k_2,k_3\in \Z$, $j_1,j_2,j_3\in \Z_+$, and
$f_{k_i,j_i}\in L^2(\R\times \R)$ are functions supported in
$\{(\xi,\tau):\xi \in [2^{k_i-1},2^{k_i+1}], \tau-\omega(\xi)\in
\wt{I}_j\}$, $i=1,2$.

(a) For any $k_1,k_2,k_3\in \Z$ and $j_1,j_2,j_3\in \Z_+$,
\begin{eqnarray}
\norm{1_{{D}_{k_3,j_3}}(\xi,\tau)(f_{k_1,j_1}*f_{k_2,j_2})}_{L^2}\les
2^{k_{min}/2}2^{j_{min}/2}\prod_{i=1}^2\norm{f_{k_i,j_i}}_{L^2}.
\end{eqnarray}

(b) For any $k_1,k_2,k_3\in \Z$ with $N_{min}\ll N_{med}\sim
N_{max}$, and $j_1,j_2,j_3\in \Z_+$, if for some $i\in \{1,2,3\}$
such that $(k_i,j_i)=(k_{min},j_{max})$ then
\begin{eqnarray}
\norm{1_{{D}_{k_3,j_3}}(\xi,\tau)(f_{k_1,j_1}*f_{k_2,j_2})}_{L^2}\les
2^{(j_{min}+j_{med})/2}2^{-\alpha
k_{max}/2}2^{-k_{min}/2}\prod_{i=1}^2\norm{f_{k_i,j_i}}_{L^2}.
\end{eqnarray}
else we have
\begin{eqnarray}
\norm{1_{{D}_{k_3,j_3}}(\xi,\tau)(f_{k_1,j_1}*f_{k_2,j_2})}_{L^2}\les
2^{(j_{min}+j_{med})/2}2^{-(1+\alpha)k_{max}/2}\prod_{i=1}^2\norm{f_{k_i,j_i}}_{L^2}.
\end{eqnarray}

(c) For any $k_1,k_2,k_3\in \Z$ with $N_{min}\sim N_{med}\sim
N_{max}\gg 1$ and $j_1,j_2,j_3\in \Z_+$
\begin{eqnarray}
\norm{1_{{D}_{k_3,j_3}}(\xi,\tau)(f_{k_1,j_1}*f_{k_2,j_2})}_{L^2}\les
2^{j_{min}/2}2^{j_{med}/4}2^{-\alpha k_{max}/4}
\prod_{i=1}^2\norm{f_{k_i,j_i}}_{L^2}.
\end{eqnarray}
\end{corollary}

\begin{proof}
Clearly, we have
\begin{eqnarray}
\norm{1_{{D}_{k_3,j_3}}(\xi,\tau)(f_{k_1,j_1}*f_{k_2,j_2})(\xi,\tau)}_{L^2}=\sup_{\norm{f}_{L^2}=1}\aabs{\int_{D_{k_3,j_3}}
f\cdot f_{k_1,j_1}*f_{k_2,j_2} d\xi d\tau}.
\end{eqnarray}
Let $f_{k_3,j_3}=1_{D_{k_3,j_3}}\cdot f$ and
$f_{k_i,j_i}^\sharp(\xi,\mu)=f_{k_i,j_i}(\xi,\mu+\omega(\xi))$,
$i=1,2,3$. Then the functions $f_{k_i,j_i}^\sharp$ are supported in
$\widetilde{I}_{k_i}\times \wt{I}_{j_i}$ and
$\norm{f_{k_i,j_i}^\sharp}_{L^2}=\norm{f_{k_i,j_i}}_{L^2}$. Using
simple changes of variables, we get that
\[\int_{D_{k_3,j_3}} f\cdot
f_{k_1,j_1}*f_{k_2,j_2} d\xi d\tau =
J(f_{k_1,j_1}^\sharp,f_{k_2,j_2}^\sharp,f_{k_3,j_3}^\sharp).\] Then
Corollary \ref{cor42} follows from Lemma \ref{lemsymes}.
\end{proof}

\begin{remark}\label{symest}
It is easy to see from the proof that Part (a) and the first
inequality in Part (b) of Lemma \ref{lemsymes} (Corollary
\ref{cor42}) also hold if we assume instead $f_{k_i,j_i}$ is
supported in $\wt{I}_{k_i}\times \wt{I}_{\leq j}$ ($D_{k_i,\leq j}$)
for $k_1,k_2,k_3\in \Z_+$.
\end{remark}

\section{Trilinear estimates}

In this section we prove Theorem \ref{thmmdg}. The ingredients are
the estimates for the characterization multiplier obtained in the
last section and $TT^*$ arguments as in \cite{Taokz}. The main
issues reduce to prove the trilinear estimates and we refer the
readers to \cite{KPV} for the other standard details.

\begin{proposition}\label{prop:trilinear}
For all $u_1,u_2,u_3$ on $\R\times \R$ and $0<\epsilon \ll 1$, we
have
\begin{eqnarray}\label{eq:trilinear}
\norm{(u_1u_2u_3)_x}_{X^{1/2-\alpha/4, -1/2+\epsilon}}\les
\prod_{j=1}^3 \norm{u_j}_{X^{1/2-\alpha/4, 1/2+\epsilon}}
\end{eqnarray}
with the implicit constant depending on $\epsilon$.
\end{proposition}

This type of estimate was systematically studied in \cite{Taokz},
see also \cite{KPV} for an elementary method. We will follow the
idea in \cite{Taokz} to prove Proposition \ref{prop:trilinear}. Let
$Z$ be any abelian additive group with an invariant measure $d\xi$.
In particular, $Z=\R^2$ in this paper. For any $k\geq 2$, Let
$\Gamma_k(Z)$ denote the hyperplane in $\R^k$
\[\Gamma_k(Z):=\{(\xi_1,\ldots, \xi_k)\in Z^k: \xi_1+\ldots+\xi_k=0\}\]
endowed with the induced measure
\[\int_{\Gamma_k(Z)}f:=\int_{Z^{k-1}}f(\xi_1,\ldots, \xi_{k-1},-\xi_1-\ldots-\xi_{k-1})d\xi_1\ldots d\xi_{k-1}.\]
Note that this measure is symmetric with respect to permutation of
the co-ordinates.

A function $m: \Gamma_k(Z)\rightarrow \C$ is said to to be a
$[k;Z]-multiplier$, and we define the norm $\norm{m}_{[k;Z]}$ to be
the best constant such that the inequality
\begin{equation}
\left|\int_{\Gamma_k(Z)}m(\xi)\prod_{j=1}^k f_i(\xi_i)\right|\leq
\norm{m}_{[k;Z]}\prod_{j=1}^k \norm{f_i}_{L^2}
\end{equation}
holds for all test functions $f_i$ on $Z$.

\begin{proof}[Proof of Proposition \ref{prop:trilinear}]
By duality, Plancherel's equality and the definition, it is easy to
see that for \eqref{eq:trilinear}, it suffices to prove
\begin{equation}\label{eq:bi}
\normo{\frac{(\xi_1 + \xi_2 + \xi_3) \langle \xi_4
\rangle^{\frac{1}{2}-\frac{\alpha}{4}}} { \jb{\tau_4 -
\omega(\xi_4)}^{1/2 - \varepsilon} \prod_{j=1}^3 \langle \xi_j
\rangle^{\frac{1}{2}-\frac{\alpha}{4}} \langle \tau_j -
\omega(\xi_j) \rangle^{1/2+\varepsilon} }}_{[4;\R \times \R]}
\lesssim 1.
\end{equation}
As in the proof of Corollary 6.3 \cite{Taokz}, we estimate $|\xi_1 +
\xi_2 + \xi_3|$ by $\langle \xi_4\rangle$. From the inequality
\[\langle \xi_4 \rangle^{\frac{3}{2}-\frac{\alpha}{4}} \lesssim
\langle \xi_4 \rangle^{1/2} \sum_{j=1}^3 \langle \xi_j
\rangle^{1-\frac{\alpha}{4}}\] and symmetry it reduces to show
\[\normo{\frac{\langle \xi_4 \rangle^{1/2} \langle \xi_2 \rangle^{1/2} } {
\langle \xi_1 \rangle^{\frac{1}{2}-\frac{\alpha}{4}} \langle \xi_3
\rangle^{\frac{1}{2}-\frac{\alpha}{4}} \jb{\tau_4 -
\omega(\xi_4)}^{1/2 - \varepsilon} \prod_{j=1}^3 \jb{\tau_j -
\omega(\xi_j)}^{1/2+\varepsilon} } }_{[4;\R \times \R]} \lesssim 1.
\]
We may minorize $\langle \tau_2 - \omega(\xi_{2})
\rangle^{1/2+\varepsilon}$ by $\langle \tau_2
-\omega(\xi_{2})\rangle^{1/2-\varepsilon}$.  But then the estimate
follows from $TT^*$ identity (Lemma 3.7, \cite{Taokz}) and the
following proposition.
\end{proof}

\begin{proposition}\label{kpv-2}
For all $u,v$ on $\R \times \R$ and $0 < \varepsilon \ll 1$, we have
$$ \| uv\|_{L^2} \lesssim
\| u\|_{X^{-1/2,1/2-\varepsilon}} \|
v\|_{X^{1/2-\alpha/4,1/2+\varepsilon}}$$
\end{proposition}

Before proving this proposition, we restate Corollary \ref{cor42} in
the following lemma.

\begin{lemma}\label{pchar}
Let $H,N_1,N_2,N_3,L_1,L_2,L_3>0$ obey \eqref{eq:dyadic1} and
\eqref{eq:dyadic2}. Then

(i) If $N_{max}\sim N_{min}$ and $L_{max}\sim
N_{max}^{1+\alpha}N_{min}$, then we have
\begin{equation}\label{eq:chari}
\eqref{eq:char} \les L_{min}^{1/2}N_{max}^{-\alpha/4}L_{med}^{1/4}.
\end{equation}

(ii) If $N_2\sim N_3 \gg N_1$ and $N_{max}^2N_{min}\sim L_1\ges
L_2,L_3$, then
\begin{equation}\label{eq:charii}
\eqref{eq:char} \les
L_{min}^{1/2}N_{max}^{-(1+\alpha)/2}\min(N_{max}^{1+\alpha}N_{min},\frac{N_{max}}{N_{min}}L_{med})^{1/2}.
\end{equation}
Similarly for permutations.

(iii) In all other cases, we have
\begin{equation}\label{eq:chariii}
\eqref{eq:char} \les
L_{min}^{1/2}N_{max}^{-(1+\alpha)/2}\min(N_{max}^{1+\alpha}N_{min},L_{med})^{1/2}.
\end{equation}
\end{lemma}

\begin{proof}[Proof of Proposition \ref{kpv-2}]

By Plancherel's equality it suffices to show that
\begin{equation}\label{kpv-2-est}
\normo{\frac{\langle \xi_2 \rangle^{1/2} } { \langle \xi_1
\rangle^{\frac{1}{2}-\frac{\alpha}{4}} \jb{\tau_2 -
\omega(\xi_2)}^{1/2 - \varepsilon} \jb{\tau_1 -
\omega(\xi_1)}^{1/2+\varepsilon} } }_{[3;\R \times \R]} \lesssim 1.
\end{equation}
By comparision principle (see \cite{Taokz}), it suffices to prove
that
\begin{eqnarray}
&&\sum_{N_1,N_2,N_3}\sum_{L_1,L_2,L_3}\sum_{H} \frac{ \langle N_2
\rangle^{1/2} }{ \langle N_1\rangle^{\frac{1}{2}-\frac{\alpha}{4}}
L_1^{1/2+\varepsilon} L_2^{1/2-\varepsilon} }\nonumber\\
&&\qquad
\qquad\norm{\chi_{N_1,N_2,N_3;H;L_1,L_2,L_3}}_{[3;\R^2]}\les 1,
\end{eqnarray}
where $N_i,L_i, H$ are dyadic numbers,
$h(\xi_1,\xi_2,\xi_3)=\omega(\xi_1)+\omega(\xi_2)+\omega(\xi_3)$ and
\begin{eqnarray}
&&\chi_{N_1,N_2,N_3;H;L_1,L_2,L_3}=\chi_{|\xi_1|\sim N_1,|\xi_2|\sim
N_2,|\xi_3|\sim N_3}\nonumber\\
&&\qquad \cdot \chi_{|h|\sim H}\chi_{|\tau_1-\omega(\xi_1)|\sim
L_1,|\tau_2-\omega(\xi_2)|\sim L_2,|\tau_3-\omega(\xi_3)|\sim L_3}.
\end{eqnarray}
The issues reduce to the estimates of
\begin{equation}\label{eq:char}
\norm{\chi_{N_1,N_2,N_3;H;L_1,L_2,L_3}}_{[3;\R^2]}
\end{equation}
and dyadic summations.

From the identity
\[\xi_1+\xi_2+\xi_3=0\]
and
\[\tau_1-\omega(\xi_1)+\tau_2-\omega(\xi_2)+\tau_3-\omega(\xi_3)+h(\xi)=0,\]
then we must have for the multiplier in \eqref{eq:char} to be
nonvanishing
\begin{eqnarray}\label{eq:dyadic1}
N_{max}&\sim& N_{med},\nonumber\\
L_{max}&\sim& \max(L_{med}, H),
\end{eqnarray}
where we define $N_{max}\geq N_{med}\geq N_{min}$ to be the maximum,
median, and minimum of $N_1,\ N_2,\ N_3$ respectively. Similarly
define $L_{max}\geq L_{med}\geq L_{min}$. It's known (see Section 4,
\cite{Taokz}) and from Lemma \ref{resoest} that we may assume
\begin{equation}\label{eq:dyadic2}
N_{max}\ges 1, \quad L_1,L_2,L_3\ges 1, \quad H\sim
N_{max}^{1+\alpha}N_{min}.
\end{equation}
Therefore, from Schur's test (Lemma 3.11, \cite{Taokz}) it suffices
to prove that
\begin{eqnarray}\label{eq:bicase1}
&&\sum_{N_{max}\sim N_{med}\sim N}\sum_{L_1,L_2,L_3\geq 1}\frac{
\langle N_2 \rangle^{1/2} }{ \langle
N_1\rangle^{\frac{1}{2}-\frac{\alpha}{4}}
L_1^{1/2+\varepsilon} L_2^{1/2-\varepsilon} }\nonumber\\
&&\qquad \qquad \times
\norm{\chi_{N_1,N_2,N_3;L_{max};L_1,L_2,L_3}}_{[3;\R^2]}
\end{eqnarray}
and
\begin{eqnarray}\label{eq:bicase2}
&&\sum_{N_{max}\sim N_{med}\sim N}\sum_{L_{max}\sim L_{med}}\sum_{H
\leq L_{max}}\frac{ \langle N_2 \rangle^{1/2} }{ \langle
N_1\rangle^{\frac{1}{2}-\frac{\alpha}{4}}
L_1^{1/2+\varepsilon} L_2^{1/2-\varepsilon} }\nonumber\\
&&\qquad \qquad \times
\norm{\chi_{N_1,N_2,N_3;H;L_1,L_2,L_3}}_{[3;\R^2]}
\end{eqnarray}
are both uniformly bounded for all $N\ges 1$.

Fix $N$.  We first prove \eqref{eq:bicase2}. By \eqref{eq:chariii}
we reduce to
\begin{align*} \sum_{ N_{max} \sim N_{med} \sim N}
&\sum_{ L_{max} \sim L_{med} \gtrsim N_{max}^{1+\alpha}N_{min}}\\
&\frac{ \langle N_2 \rangle^{1/2} }{ \langle
N_1\rangle^{\frac{1}{2}-\frac{\alpha}{4}} L_1^{1/2+\varepsilon}
L_2^{1/2-\varepsilon} } L_{min}^{1/2}
N_{max}^{-\frac{1+\alpha}{2}}L_{med}^{1/2}  \lesssim 1.
\end{align*}
Using the estimate
$$ \frac{\langle N_2 \rangle^{1/2} } {\langle N_1 \rangle^{\frac{1}{2}-\frac{\alpha}{4}} }
\lesssim \frac{N^{1/2}}{\langle N_{min}
\rangle^{\frac{1}{2}-\frac{\alpha}{4}}}; \quad L_1^{1/2+\varepsilon}
L_2^{1/2-\varepsilon} \gtrsim L_{min}^{1/2+\varepsilon}
L_{med}^{1/2-\varepsilon} $$ and then performing the $L$ summations,
we reduce to
$$\sum_{ N_{max} \sim N_{med} \sim N}
\frac{ N^{-\frac{\alpha}{2}}(N_{max}^{1+\alpha}N_{min})^\varepsilon
}{ \langle N_{min}\rangle^{\frac{1}{2}-\frac{\alpha}{4}} } \lesssim
1.$$ which is certainly true since $0<\alpha \leq 1$.

Now we show \eqref{eq:bicase1}.  We may assume $L_{max} \sim
N_{max}^{1+\alpha} N_{min}$.

We assume first $N_{max}\sim N_{min}\sim N$. In this case applying
\eqref{eq:chari} we reduce to
$$
\sum_{ L_{max} \sim N^3} \frac{N^{1/2} }{
N^{\frac{1}{2}-\frac{\alpha}{4}} L_{min}^{1/2+\varepsilon}
L_{med}^{1/2-\varepsilon} } L_{min}^{1/2}
N_{max}^{-\frac{\alpha}{4}} L_{med}^{1/4}  \lesssim 1
$$
which is easily verified. Now we assume $N_{max}\sim N_{med}\gg
N_{min}$ where \eqref{eq:charii} applies. We have three cases
\begin{align*}
 N \sim N_1 \sim N_2 \gg N_3&; H \sim L_3 \gtrsim L_1, L_2\\
 N \sim N_2 \sim N_3 \gg N_1&; H \sim L_1 \gtrsim L_2, L_3\\
 N \sim N_1 \sim N_3 \gg N_2&; H \sim L_2 \gtrsim L_1, L_3
\end{align*}

In the first case we reduce to
\begin{eqnarray*}
&&\sum_{N_3 \ll N } \sum_{1 \lesssim L_1,L_2 \lesssim N^{1+\alpha}
N_3} \frac{ N^{1/2} }{ N^{\frac{1}{2}-\frac{\alpha}{4}}
L_1^{1/2+\varepsilon} L_2^{1/2-\varepsilon} }\\
&&\quad L_{min}^{1/2}N_{max}^{-(1+\alpha)/2}
\min(N_{max}^{1+\alpha}N_{min},\frac{N_{max}}{N_{min}}L_{med})^{1/2}.
\end{eqnarray*}
Performing the $N_3$ summation we reduce to
$$
\sum_{1 \lesssim L_1,L_2 \lesssim N^3} \frac{ N^{1/2} }{
N^{\frac{1}{2}-\frac{\alpha}{4}} L_1^{1/2+\varepsilon}
L_2^{1/2-\varepsilon} } L_{min}^{1/2} N^{-\frac{\alpha}{4}}
L_{med}^{1/4} \lesssim 1$$ which is easily verified.

To unify the second and third cases we replace
$L_1^{1/2+\varepsilon}$ by $L_1^{1/2-\varepsilon}$.  It suffices now
to show the second case. We simplify using \eqref{eq:charii} to
$$
\sum_{N_1 \ll N } \sum_{1 \lesssim L_2,L_3 \ll N^2 N_1}
\frac{N^{1/2}}{ \langle N_1\rangle^{\frac{1}{2}-\frac{\alpha}{4}}
(N^{1+\alpha} N_1)^{1/2-\varepsilon} L_2^{1/2-\varepsilon} }
L_{min}^{1/2} N_1^{1/2} \lesssim 1.
$$
We may assume $N_1 \gtrsim N^{-(1+\alpha)}$ since the inner sum
vanishes otherwise.  Performing the $L$ summation we reduce to
$$
\sum_{N^{-(1+\alpha)} \lesssim N_1 \ll N } \frac{N^{1/2}}{ \langle
N_1\rangle^{\frac{1}{2}-\frac{\alpha}{4}} (N^{1+\alpha}
N_1)^{1/2-2\varepsilon} } N_1^{1/2} \lesssim 1
$$
which is easily verified (with about $N^{-\frac{\alpha}{2}}$ to
spare).

To finish the proof of \eqref{eq:bicase1} it remains to deal with
the cases where \eqref{eq:chariii} holds.  This reduces to
$$
\sum_{ N_{max} \sim N_{med} \sim N} \sum_{ L_{max} \sim
N_{max}^{1+\alpha} N_{min}} \frac{ \langle N_2 \rangle^{1/2} }{
\langle N_1\rangle^{\frac{1}{2}-\frac{\alpha}{4}}
L_1^{1/2+\varepsilon} L_2^{1/2-\varepsilon} } L_{min}^{1/2}
N_{max}^{-\frac{1+\alpha}{2}}L_{med}^{1/2}  \lesssim 1.
$$
Performing the $L$ summations, we reduce to
$$
\sum_{ N_{max} \sim N_{med} \sim N} \frac{ \langle N_2 \rangle^{1/2}
(N_{max}^{1+\alpha} N_{min})^{\varepsilon}}{ \langle
N_1\rangle^{\frac{1}{2}-\frac{\alpha}{4}} }
N_{max}^{-\frac{1+\alpha}{2}} \lesssim 1$$ which is easily verified.
\end{proof}

We see from the proof that $\alpha>0$ plays crucial roles. The
implicit constant in \eqref{eq:trilinear} depends on both $\alpha$
and $\epsilon$.

\section{Short-time bilinear estimates}

We prove now some dyadic bilinear estimates. We will need an
estimate on the resonance. For its proof we refer the reader to
Lemma 3.8 in \cite{Herr2}.
\begin{lemma}\label{resoest}
Let $0\leq \alpha \leq 1$. Then
\[|\Omega(\xi_1,\xi_2)|\sim |\xi|_{max}^{1+\alpha}|\xi|_{min},\]
where
\[|\xi|_{max}=\max(|\xi_1|,|\xi_2|,|\xi_1+\xi_2|),\quad |\xi|_{min}=\min(|\xi_1|,|\xi_2|,|\xi_1+\xi_2|).\]
\end{lemma}

\begin{proposition}[high-low]\label{phl} If $k_3\geq 20$, $|k_2-k_3|\leq 5$, $0\leq k_1\leq k_2-10$, then
\begin{eqnarray}\label{eq:p711}
\norm{P_{k_3}\partial_x(u_{k_1}v_{k_2})}_{N_{k_3}}\les
\norm{u_{k_1}}_{F_{k_1}}\norm{v_{k_2}}_{F_{k_2}}.
\end{eqnarray}
\end{proposition}

\begin{proof}
Using the definitions and \eqref{eq:pXk3}, we obtain that the
left-hand side of \eqref{eq:p711} is dominated by
\begin{eqnarray}\label{eq:p712}
&&C\sup_{t_k\in
\R}\norm{(\tau-\omega(\xi)+i2^{[(1-\alpha)k_3]})^{-1}\cdot
2^{k_3}1_{I_{k_3}}(\xi) \nonumber\\
&&\quad \cdot
\ft[u_{k_1}\eta_0(2^{[(1-\alpha)k_3]-2}(t-t_k))]*\ft[v_{k_2}
\eta_0(2^{[(1-\alpha)k_3]-2}(t-t_k))]}_{X_k}.
\end{eqnarray}
To prove Proposition \ref{phl}, it suffices to prove that if
$j_i\geq [(1-\alpha)k_3]$ and $f_{k_i,j_i}: \R^2\rightarrow \R_+$
are supported in $\widetilde{D}_{k_i,j_i}$ for $i=1,2$, then
\begin{eqnarray}\label{eq:p713}
&&2^{k_3}\sum_{j_3\geq[(1-\alpha)k_3]}2^{-j_3/2}\norm{1_{\widetilde{D}_{k_3,j_3}}\cdot
(f_{k_1,j_1}*f_{k_2,j_2})}_{L^2}\nonumber\\
&&\les 2^{(s_\alpha-1/4) k_1}
2^{(j_1+j_2)/2}\norm{f_{k_1,j_1}}_{L^2}\norm{f_{k_2,j_2}}_{L^2}.
\end{eqnarray}

Indeed, let $f_{k_1}=\ft[u_{k_1}\eta_0(2^{(1-\alpha)k_3-2}(t-t_k))]$
and $f_{k_2}=\ft[v_{k_2}\eta_0(2^{(1-\alpha)k_3-2}(t-t_k))]$. Then
from the definition of $X_k$ we get that \eqref{eq:p712} is
dominated by
\begin{eqnarray}\label{eq:p714}
\sup_{t_k\in
\R}2^{k_3}\sum_{j_3=0}^{\infty}2^{j_3/2}\sum_{j_1,j_2\geq
[(1-\alpha)k_3]}\norm{(2^{j_3}+i2^{(1-\alpha)k_3})^{-1}1_{{D}_{k_3,j_3}}\cdot
f_{k_1,j_1}*f_{k_2,j_2}}_{L^2},
\end{eqnarray}
where we set
$f_{k_i,j_i}=f_{k_i}(\xi,\tau)\eta_{j_i}(\tau-\omega(\xi))$ for
$j_i>[(1-\alpha)k_3]$ and the remaining part
$f_{k_i,[(1-\alpha)k_3]}=f_{k_i}(\xi,\tau)\eta_{\leq
[(1-\alpha)k_3]}(\tau-\omega(\xi))$, $i=1,2$. For the summation on
the terms $j_3<[(1-\alpha)k_3]$ in \eqref{eq:p714}, we get from the
fact $1_{D_{k_3,j_3}}\leq 1_{\wt{D}_{k_3,j_3}}$ that
\begin{eqnarray}
&&\sup_{t_k\in
\R}2^{k_3}\sum_{j_3<[(1-\alpha)k_3]}2^{j_3/2}\sum_{j_1,j_2\geq
[(1-\alpha)k_3]}\norm{(2^{j_3}+i2^{(1-\alpha)k_3})^{-1}1_{{D}_{k_3,j_3}}\cdot
f_{k_1,j_1}*f_{k_2,j_2}}_{L^2}\nonumber\\
&&\les \sup_{t_k\in \R}2^{k_3}\sum_{j_1,j_2\geq
[(1-\alpha)k_3]}2^{-[(1-\alpha)k_3]/2}\norm{1_{\wt{D}_{k_3,[(1-\alpha)k_3]}}\cdot
f_{k_1,j_1}*f_{k_2,j_2}}_{L^2}.
\end{eqnarray}
From the fact that $f_{k_i,j_i}$ is supported in $\wt{D}_{k_i,j_i}$
for $i=1,2$ and using \eqref{eq:p713}, then we get that
\begin{eqnarray*}
&&\sup_{t_k\in \R}2^{k_3}\sum_{j_1,j_2\geq
[(1-\alpha)k_3]}2^{-[(1-\alpha)k_3]/2}\norm{1_{\wt{D}_{k_3,[(1-\alpha)k_3]}}\cdot
f_{k_1,j_1}*f_{k_2,j_2}}_{L^2}\\
 &&\les 2^{(s_\alpha-1/4) k_1}\sup_{t_k\in \R}
\sum_{j_1,j_2\geq
[(1-\alpha)k_3]}2^{j_1/2}\norm{f_{k_1,j_1}}_{L^2}2^{j_2/2}\norm{f_{k_2,j_2}}_{L^2}.
\end{eqnarray*}
Thus from the definition and using \eqref{eq:pXk2} and
\eqref{eq:pXk3} we obtain \eqref{eq:p711}, as desired. To prove
\eqref{eq:p713}, we apply Corollary \ref{cor42} (b) and Remark
\ref{symest} that
\begin{eqnarray*}
&&2^{k_3}\sum_{j_3\geq
[(1-\alpha)k_3]}2^{-j_3/2}\norm{1_{\widetilde{D}_{k_3,j_3}}\cdot
(f_{k_1,j_1}*f_{k_2,j_2})}_{L^2} \nonumber\\
&&\les 2^{k_3}\sum_{j_3\geq
[(1-\alpha)k_3]}2^{(j_1+j_2-j_3)/2}2^{-(1+\alpha)k_3/2}\prod_{i=1}^2\norm{f_{k_i,j_i}}_{L^2}\les
2^{(j_1+j_2)/2}\prod_{i=1}^2\norm{f_{k_i,j_i}}_{L^2}.
\end{eqnarray*}
Therefore, we complete the proof of the proposition.
\end{proof}

We see from the proof that if we only consider the interactions in
short time, then the modulation has a bound below, thus we are able
to control the high low interactions in time interval of length
$2^{-[(1-\alpha)k]}$.

\begin{proposition}\label{phhh}
Assume $k_3\geq 20$. If $|k_3-k_2|\leq 5$ and $|k_1-k_2|\leq 5$ then
we have
\begin{eqnarray}
\norm{P_{k_3}\partial_x(u_{k_1}v_{k_2})}_{N_{k_3}}\les
 \norm{u_{k_1}}_{F_{k_1}}\norm{v_{k_2}}_{F_{k_2}}.
\end{eqnarray}
\end{proposition}
\begin{proof}
As in the proof of Proposition \ref{phl}, it suffices to prove that
if $j_1,j_2\geq [(1-\alpha)k_3]$ and $f_{k_i,j_i}: \R^2\rightarrow
\R_+$ are supported in $\widetilde{D}_{k_i,j_i}$, $i=1,2$, then
\begin{eqnarray}\label{eq:p722}
2^{k_3}\sum_{j_3\geq[(1-\alpha)k_3]}2^{-j_3/2}\norm{1_{\widetilde{D}_{k_3,j_3}}
(f_{k_1,j_1}*f_{k_2,j_2})}_{L^2}\les
2^{j_1/2}\norm{f_{k_1,j_1}}_{L^2}\cdot2^{j_2/2}\norm{f_{k_2,j_2}}_{L^2}.
\end{eqnarray}
Since by Lemma \ref{resoest} we get in the area $\{|\xi_i| \in
\widetilde{I}_{k_i},i=1,2\}\cap \{|\xi_1+\xi_2|\in
\widetilde{I}_{k_3}\}$
\[|\Omega(\xi_1,\xi_2)|\sim 2^{(2+\alpha)k_3},\]
then by checking the support properties, we get
$1_{\widetilde{D}_{k_3,j_3}}\cdot (f_{k_1,j_1}*f_{k_2,j_2})\equiv 0$
unless $j_{max}\geq (2+\alpha)k_3-30$. Then it follows from
Corollary \ref{cor42} (a) that the left-hand side of \eqref{eq:p722}
is bounded by
\begin{eqnarray}
2^{k_3}\sum_{j_3\geq
[(1-\alpha)k_3]}2^{-j_3/2}2^{j_{min}/2}2^{k_{min}/2}\prod_{i=1}^2\norm{f_{k_i,j_i}}_{L^2}.
\end{eqnarray}
Then we get the bound \eqref{eq:p722} by considering either
$j_3=j_{max}$ or $j_3\ne j_{max}$.
\end{proof}

\begin{proposition}\label{phhl}
If $k_2\geq 20$, $|k_1-k_2|\leq 5$ and $ 0\leq k_3\leq k_1-10$, then
we have
\begin{eqnarray}\label{eq:p741}
\norm{P_{k_3}\partial_x(u_{k_1}v_{k_2})}_{N_{k_3}}\les k_2^2
2^{-(1-\alpha)k_3}2^{(\frac{1}{2}-2\alpha)k_2}
\norm{u_{k_1}}_{F_{k_1}}\norm{v_{k_2}}_{F_{k_2}}.
\end{eqnarray}
\end{proposition}
\begin{proof}
Let $\beta:\R\rightarrow [0,1]$ be a smooth function supported in
$[-1,1]$ with the property that
\[\sum_{n\in \Z}\beta^2(x-n)\equiv 1, \quad x\in \R.\]
Using the definitions, the left-hand side of \eqref{eq:p741} is
dominated by
\begin{eqnarray*}
&&C\sup_{t_k\in \R}\sum_{k_3'\leq
k_3}\normb{(\tau-\omega(\xi)+i2^{(1-\alpha){k_3'}_+})^{-1}2^{k_3'}\chi_{k_3'}(\xi)\sum_{|m|\leq
C2^{(1-\alpha)(k_2-{k_3'}_+)}}\\
&& \qquad \ft[u_{k_1}\eta_0(2^{(1-\alpha){k_3'}_+}(t-t_k))\beta(2^{(1-\alpha)k_2}(t-t_k)-m)]*\\
&& \qquad
\ft[u_{k_2}\eta_0(2^{(1-\alpha){k_3'}_+}(t-t_k))\beta(2^{(1-\alpha)k_2}(t-t_k)-m)]}_{X_k}.
\end{eqnarray*}

We assume first $k_3=0$. In view of the definitions, \eqref{eq:pXk2}
and \eqref{eq:pXk3}, it suffices to prove that if $j_1,j_2\geq
[(1-\alpha)k_2]$, and $f_{k_i,j_i}:\R^3 \ra \R_+$ are supported in
$\widetilde{D}_{k_i,j_i}$, $i=1,2$, then
\begin{eqnarray}\label{eq:hhl1}
&&\sum_{k_3'\leq 0} 2^{k_3'}2^{(1-\alpha)k_2}\sum_{j_3\geq 0}
2^{-j_3/2}\norm{\chi_{k_3'}(\xi)\eta_{\leq j_3}(\tau-\omega(\xi))(f_{k_1,j_1}*f_{k_2,j_2})}_{L^2}\nonumber\\
&& \les k_2^2 2^{(\frac{1}{2}-2\alpha)k_2}
2^{j_1/2}\norm{f_{k_1,j_1}}_{L^2}\cdot
2^{j_2/2}\norm{f_{k_2,j_2}}_{L^2}
\end{eqnarray}

To prove \eqref{eq:hhl1}, we may assume $k_3'\geq -10k_2$, since
otherwise we use Corollary \ref{cor42} (a). From Lemma \ref{resoest}
and the supports properties as in Proposition \ref{phhh}, we get
$j_{max}\geq (1+\alpha)k_2+k_3'-30$. Then it follows from Corollary
\ref{cor42} (b) that the left-hand side of \eqref{eq:hhl1} is
bounded by
\begin{eqnarray*}
&&Ck_2\sum_{k_3'\leq 0}
2^{k_3'}2^{(1-\alpha)k_2}2^{-(1+\alpha)k_2/2}2^{-k_3'/2}2^{j_1/2}2^{j_2/2}2^{-\alpha
k_2/2}2^{-k_3'/2}\norm{f_{k_1,j_1}}_{L^2}\norm{f_{k_2,j_2}}_{L^2}\\
&&\les k_2^2
2^{(\frac{1}{2}-2\alpha)k_2}2^{j_1/2}2^{j_2/2}\norm{f_{k_1,j_1}}_{L^2}\norm{f_{k_2,j_2}}_{L^2}
\end{eqnarray*}

We assume now $k_3\geq 1$. It suffices to prove that if $j_1,j_2\geq
[(1-\alpha)k_2]$, and $f_{k_i,j_i}:\R^3 \ra \R_+$ are supported in
$\widetilde{D}_{k_i,j_i}$, $i=1,2$, then
\begin{eqnarray}
&& 2^{k_3}2^{(1-\alpha)(k_2-k_3)}\sum_{j_3\geq (1-\alpha)k_3}
2^{-j_3/2}\norm{\chi_{k_3}(\xi)\eta_{\leq j_3}(\tau-\omega(\xi))(f_{k_1,j_1}*f_{k_2,j_2})}_{L^2}\nonumber\\
&& \les k_2^2 2^{-(1-\alpha)k_3}2^{(\frac{1}{2}-2\alpha)k_2}
2^{j_1/2}\norm{f_{k_1,j_1}}_{L^2}\cdot
2^{j_2/2}\norm{f_{k_2,j_2}}_{L^2}
\end{eqnarray}
which can be proved similarly as \eqref{eq:hhl1}.
\end{proof}

\begin{proposition}\label{plll}
If $0\leq k_1,k_2,k_3\leq 200$, then
\begin{eqnarray}
\norm{P_{k_3}\partial_x(u_{k_1}v_{k_2})}_{N_{k_3}}\les
\norm{u_{k_1}}_{F_{k_1}}\norm{v_{k_2}}_{F_{k_2}}.
\end{eqnarray}
\end{proposition}
\begin{proof}
This follows immediately from the definitions, Corollary \ref{cor42}
(a), Remark \ref{symest} and \eqref{eq:pXk2} and \eqref{eq:pXk3}.
\end{proof}

As a conclusion to this section we prove the bilinear estimates,
using the dyadic bilinear estimates obtained above.
\begin{lemma}\label{pbilinear}
(a) If $s\geq 1-\alpha$, $T\in (0,1]$, and $u,v\in F^{s}(T)$ then
\begin{eqnarray}
\norm{\partial_x(uv)}_{N^{s}(T)}&\les&
\norm{u}_{F^{s}(T)}\norm{v}_{F^{1-\alpha}(T)}+\norm{u}_{F^{1-\alpha}(T)}\norm{v}_{F^{s}(T)}.
\end{eqnarray}

(b)If $T\in (0,1]$, $u\in F^{0}(T)$ and $v\in F^{s_\alpha}(T)$ then
\begin{eqnarray}
\norm{\partial_x(uv)}_{N^{0}(T)}&\les&
\norm{u}_{F^{0}(T)}\norm{v}_{F^{1-\alpha}(T)}.
\end{eqnarray}
\end{lemma}
\begin{proof}
Since $P_kP_j=0$ if $k\neq j$ and $k,j\in \Z_+$, then we can fix
extensions $\wt{u}, \wt{v}$ of $u, v$ such that $\norm{P_k(\wt
u)}_{F_k}\leq 2\norm{P_k(u)}_{F_k(T)}$ and $\norm{P_k(\wt
v)}_{F_k}\leq 2\norm{P_k(v)}_{F_k(T)}$ for any $k\in \Z_+$. In view
of definition, we get
\begin{eqnarray*}
\norm{\partial_x(u v
)}_{N^{s}(T)}^2\les\sum_{k_3=0}^{\infty}2^{2sk_3}\norm{P_{k_3}(\partial_x(\wt
u \wt v))}_{N_{k_3}}^2.
\end{eqnarray*}
For $k\in \Z_+$ let $\wt{u}_{k}=P_{k}(\wt u)$ and
$\wt{v}_{k}=P_{k}(\wt v)$, then we get
\begin{eqnarray}\label{eq:bilineares1}
\norm{P_{k_3}(\partial_x(\wt u \wt v))}_{N_{k_3}}\les
\sum_{k_1,k_2\in \Z_+}\norm{P_{k_3}(\partial_x(\wt u_{k_1} \wt
v_{k_2}))}_{N_{k_3}}.
\end{eqnarray}
From symmetry we may assume $k_1\leq k_2$. Dividing the summation on
the right-hand side of \eqref{eq:bilineares1} into several parts, we
get
\begin{eqnarray}
\sum_{k_1,k_2\in \Z_+}\norm{P_{k_3}(\partial_x(\wt u_{k_1} \wt
v_{k_2}))}_{N_{k_3}}&\les& \sum_{i=1}^4
\sum_{{A_i}}\norm{P_{k_3}(\partial_x(\wt u_{k_1} \wt
v_{k_2}))}_{N_{k_3}}
\end{eqnarray}
where we denote
\begin{eqnarray*}
&&A_1=\{k_1\leq k_2: |k_2-k_3|\leq 5, k_1\leq k_2-10, \mbox{ and }
k_2\geq
20\};\\
&&A_2=\{k_1\leq k_2: |k_2-k_3|\leq 5, |k_1-k_2|\leq 10, \mbox{ and }
k_2\geq
20\};\\
&&A_3=\{k_1\leq k_2: k_3\leq k_2-10, |k_1-k_2|\leq 5, \mbox{ and }
k_1 \geq 20\}.\\
&&A_4=\{k_1\leq k_2: k_1,k_2,k_3\leq 200\}.
\end{eqnarray*}

For part (a), it suffices to prove that for $i=1,2,3,4$ then
\begin{eqnarray}
\normo{2^{sk_3}\sum_{{A_i}}\norm{P_{k_3}(\partial_x(\wt u_{k_1} \wt
v_{k_2}))}_{N_{k_3}}}_{l^2_{k_3}}\les \norm{\wt u}_{F^{s}}\norm{\wt
v}_{F^{s_\alpha}}+\norm{\wt u}_{F^{s_\alpha}}\norm{\wt v}_{F^{s}},
\end{eqnarray}
which follows from Proposition \ref{phl}-\ref{plll}. For part (b),
it suffices to prove
\begin{eqnarray}\label{eq:bib}
\normo{\sum_{{k_1,k_2\in \Z_+}}\norm{P_{k_3}(\partial_x(\wt u_{k_1}
\wt v_{k_2}))}_{N_{k_3}}}_{l^2_{k_3}}\les \norm{\wt
v}_{F^{0}}\norm{\wt u}_{F^{s_\alpha}}.
\end{eqnarray}
Similarly we divide the summation on the left-hand side of
\eqref{eq:bib} into many pieces, but now we do not have symmetries.
We denote for $i=1,2,3,4$
\[\bar{A_i}=\{(k_1,k_2): (k_2,k_1)\in
A_i\}.\] For the summation in $\bar{A_1}\cup A_1$ we can get easily
control it using Proposition \ref{phl}. The contributions of the
summation in $\bar{A_2}\cup A_2$ and $\bar{A_4}\cup A_4$ are
acceptable due to Proposition \ref{phhh} and \ref{plll}. For the
summation in $\bar{A_3}\cup A_3$ we use Proposition \ref{phhl} since
for $0\leq \alpha \leq 1$ we have $1-\alpha>1/2-2\alpha$.
\end{proof}

\section{Proof of Theorem \ref{thmmain}}

In this section we devote to prove Theorem \ref{thmmain}. The main
ingredients are energy estimates which is proved in the next section
and short-time bilinear estimates obtained in the last section. The
idea is due to Ionescu, Kenig and Tataru \cite{IKT}.

\begin{proposition}\label{pFstoHs}
Let $s\geq 0$, $T\in (0,1]$, and $u\in F^{s}(T)$, then
\begin{equation}
\sup_{t\in [-T,T]}\norm{u(t)}_{{H}^s}\les\ \norm{u}_{F^{s}(T)}.
\end{equation}
\end{proposition}
\begin{proof}
 In view of the definitions, it suffices to prove that if $k\in
\Z_+$, $t_k\in [-1,1]$, and $\wt u_k \in F_k$ then
\begin{equation}
\norm{\ft[\wt u_k(t_k)]}_{L_\xi^2}\les \norm{\ft[\wt u_k\cdot
\eta_0(2^{[(1-\alpha)k]}(t-t_k))]}_{X_k}.
\end{equation}
Let $f_k=\ft[\wt u_k\cdot \eta_0(2^{[(1-\alpha)k]}(t-t_k))]$, so
\[\ft[\wt u_k(t_k)](\xi)=c\int_\R f_k(\xi,\tau)e^{it_k\tau}d\tau.\]
From the definition of $X_k$, we get that
\[ \norm{\ft[\wt
u_k(t_k)]}_{L_\xi^2}\les \normo{\int_\R
|f_k(\xi,\tau)|d\tau}_{L_\xi^2}\les \norm{f_k}_{X_k},\] which
completes the proof of the proposition.
\end{proof}

\begin{proposition}\label{plinear}
Assume $T\in (0,1]$, $u,v\in C([-T,T]:H^\infty)$ and
\begin{equation}\label{eq:lBO}
u_t+|\partial_x|^{1+\alpha}\partial_x u=v \mbox{ on } \R\times
(-T,T).
\end{equation}
Then for any $s\geq 0$,
\begin{equation}\label{eq:p921}
\norm{u}_{F^{s}(T)}\les \ \norm{u}_{E^{s}(T)}+\norm{v}_{N^{s}(T)}.
\end{equation}
\end{proposition}
\begin{proof}
In view of the definitions, we see that the square of the right-hand
side of \eqref{eq:p921} is equivalent to
\begin{eqnarray*}
&&\norm{P_{\leq 0}(u(0))}_{L^2}^2+\norm{P_{\leq
0}(v)}_{N_k(T)}^2\\
&&+\sum_{k\geq 1}\big(\sup_{t_k\in
[-T,T]}2^{2sk}\norm{P_k(u(t_k))}_{L^2}^2
+2^{2sk}\norm{P_k(v)}_{N_k(T)}^2\big).
\end{eqnarray*}
Thus, from definitions, it suffices to prove that if $k\in \Z_+$ and
$u,v \in C([-T,T]:H^\infty)$ solve \eqref{eq:lBO}, then
\begin{eqnarray}\label{eq:retardlinear}
\left\{\begin{array}{l} \norm{P_{\leq 0}(u)}_{F_0(T)}\les
\norm{P_{\leq 0}(u(0))}_{L^2}+\norm{P_{\leq 0}(v)}_{N_0(T)};\\
\norm{P_k(u)}_{F_k(T)}\les \sup_{t_k\in
[-T,T]}\norm{P_k(u(t_k))}_{L^2}+\norm{P_k(v)}_{N_k(T)} \mbox{ if }
k\geq 1.
\end{array}
\right.
\end{eqnarray}

We only prove the second inequality in \eqref{eq:retardlinear},
since the first one can be treated in the same ways. Fix $k\geq 1$
and let $\wt v$ denote an extension of $P_k(v)$ such that $\norm{\wt
v}_{N_k}\leq C\norm{v}_{N_k(T)}$. In view of \eqref{eq:Sk}, we may
assume that $\wt v$ is supported in $\R\times
[-T-2^{-[(1-\alpha)k]-10},T+2^{-[(1-\alpha)k]-10}]$. Indeed, let
$\theta(t)$ be a smooth function such that
\[\theta(t)=1, \mbox{ if }t\geq 1;\quad \theta(t)=0, \mbox{ if } t\leq 0.\]
Thus $\theta(2^{[(1-\alpha)k]+10}(t+T+2^{-[(1-\alpha)k]-10}))$,
$\theta(-2^{[(1-\alpha)k]+10}(t-T-2^{-[(1-\alpha)k]-10})) \in S_k$.
Then we see that
$\theta(2^{[(1-\alpha)k]+10}(t+T+2^{-[(1-\alpha)k]-10}))\theta(-2^{[(1-\alpha)k]+10}(t-T-2^{-[(1-\alpha)k]-10}))$
is supported in $[-T-2^{-[(1-\alpha)k]-10},T+2^{-[(1-\alpha)k]-10}]$
and equal to $1$ in $[-T,T]$. From \eqref{eq:Sk} we consider
$\wt{v}\theta(2^{k+10}(t+T+2^{-k-10}))\theta(-2^{k+10}(t-T-2^{-k-10}))$
instead. For $t\geq T$ we define
\[\wt u (t)=\eta_0(2^{[(1-\alpha)k]+5}(t-T))\big[W(t-T)P_k(u(T))+\int_T^tW(t-s)(P_k(\wt v(s)))ds \big].\]
For $t\leq -T$ we define
\[\wt u (t)=\eta_0(2^{[(1-\alpha)k]+5}(t+T))\big[W(t+T)P_k(u(-T))+\int_{-T}^tW(t-s)(P_k(\wt v(s)))ds \big].\]
For $t\in [-T,T]$ we define $\wt u(t)=u(t)$. It is clear that $\wt
u$ is an extension of u and we get from \eqref{eq:Sk} that
\begin{eqnarray}\label{eq:extu}
\norm{u}_{F_k(T)}\les \sup_{t_k\in [-T,T]}\norm{\ft[\wt u \cdot
\eta_0(2^{[(1-\alpha)k]}(t-t_k))]}_{X_k}.
\end{eqnarray}
Indeed, to prove \eqref{eq:extu}, it suffices to prove that
\begin{eqnarray}
\sup_{t_k\in \R}\norm{\ft[\wt u \cdot
\eta_0(2^{[(1-\alpha)k]}(t-t_k))]}_{X_k}\les \sup_{t_k\in
[-T,T]}\norm{\ft[\wt u \cdot
\eta_0(2^{[(1-\alpha)k]}(t-t_k))]}_{X_k}.
\end{eqnarray}
For $t_k>T$, since $\wt{u}$ is supported in
$[-T-2^{-[(1-\alpha)k]-5},T+2^{-[(1-\alpha)k]-5}]$, it is easy to
see that
\[\wt{u}\eta_0(2^{[(1-\alpha)k]}(t-t_k))=\wt{u}\eta_0(2^{[(1-\alpha)k]}(t-T))\eta_0(2^{[(1-\alpha)k]}(t-t_k)).\]
Therefore, we get from \eqref{eq:pXk3} that
\[\sup_{t_k>T}\norm{\ft[\wt u \cdot
\eta_0(2^{k}(t-t_k))]}_{X_k}\les \sup_{t_k\in [-T,T]}\norm{\ft[\wt u
\cdot \eta_0(2^{k}(t-t_k))]}_{X_k}.\] Using the same method for
$t_k<-T$, we obtain \eqref{eq:extu} as desired.

Now we prove the second inequality in \eqref{eq:retardlinear}. In
view of the definitions, \eqref{eq:extu} and \eqref{eq:pXk3}, it
suffices to prove that if $\phi_k \in L^2$ with $\widehat{\phi_k}$
supported in $I_k$, and $v_k\in N_k$ then
\begin{eqnarray}
\norm{\ft[u_k\cdot \eta_0(2^{[(1-\alpha)k]}t)]}_{X_k}\les
\norm{\phi_k}_{L^2}+\norm{(\tau-\omega(\xi)+i2^{[(1-\alpha)k]})^{-1}\cdot
\ft(v_k)}_{X_k},
\end{eqnarray}
where
\begin{equation}
u_k(t)=W(t)(\phi_k)+\int_0^tW(t-s)(v_k(s))ds.
\end{equation}
Straightforward computations show that
\begin{eqnarray*}
&&\ft[u_k\cdot
\eta_0(2^{[(1-\alpha)k]}t)](\xi,\tau)=\widehat{\phi_k}(\xi)\cdot
2^{-[(1-\alpha)k]}\widehat{\eta_0}(2^{-[(1-\alpha)k]}(\tau-\omega(\xi)))\\
&&+C\int_\R \ft(v_k)(\xi,\tau')\cdot
\frac{\widehat{\eta_0}(2^{-[(1-\alpha)k]}(\tau-\tau'))-\widehat{\eta_0}(2^{-[(1-\alpha)k]}(\tau-\omega(\xi)))}{2^{[(1-\alpha)k]}(\tau'-\omega(\xi))}d\tau'.
\end{eqnarray*}
We observe now that
\begin{eqnarray*}
&&\aabs{\frac{\widehat{\eta_0}(2^{-[(1-\alpha)k]}(\tau-\tau'))-\widehat{\eta_0}(2^{-[(1-\alpha)k]}(\tau-\omega(\xi)))}{2^{[(1-\alpha)k]}(\tau'-\omega(\xi))}\cdot
(\tau'-\omega(\xi)+i2^{[(1-\alpha)k]})}\\
&&\les \
2^{-[(1-\alpha)k]}(1+2^{-[(1-\alpha)k]}|\tau-\tau'|)^{-4}+2^{-[(1-\alpha)k]}(1+2^{-[(1-\alpha)k]}|\tau-\omega(\xi)|)^{-4}.
\end{eqnarray*}
Using \eqref{eq:pXk1} and \eqref{eq:pXk2}, we complete the proof of
the proposition.
\end{proof}

Now we turn to prove Theorem \ref{thmmain}. To prove Theorem
\ref{thmmain} (a), by the scaling \eqref{eq:scaling} we may assume
that
\begin{equation}\label{eq:smalldata}
\norm{u_0}_{H^s}\leq \epsilon\ll 1.
\end{equation}
The uniqueness follows from the classical energy methods. We only
need to construct the solution on the time interval $[-1,1]$. In
view of the classical results, it suffices to prove that if $T\in
(0,1]$ and $u\in C([-T,T]:H^\infty)$ is a solution of
\eqref{eq:dgBO} with $\norm{u_0}_{H^s}\leq \epsilon\ll 1$ then
\begin{eqnarray}\label{eq:H2est}
\sup_{t\in [-T,T]}\norm{u(t)}_{H^2}\les \norm{u_0}_{H^2}.
\end{eqnarray}

It follows from Proposition \ref{plinear}, Proposition
\ref{pbilinear} and the energy estimate Proposition \ref{penergy}
that for any $T'\in [0,T]$ we have
\begin{eqnarray}\label{eq:Fsest}
\left \{
\begin{array}{l}
\norm{u}_{F^{s}(T')}\les \norm{u}_{E^{s}(T')}+\norm{\partial_x(u^2)}_{N^{s}(T')};\\
\norm{\partial_x(u^2)}_{N^{s}(T')}\les \norm{u}_{F^{s}(T')}^2;\\
\norm{u}_{E^{s}(T')}^2\les
\norm{\phi}_{{H}^s}^2+\norm{u}_{F^{s}(T')}^3.
\end{array}
\right.
\end{eqnarray}
We denote
$X(T')=\norm{u}_{E^s(T')}+\norm{\partial_x(u^2)}_{N^{s}(T')}$. Then
by a similar argument as in the proof of Lemma 4.2 in \cite{IKT}, we
know $X(T')$ is continuous and satisfies
\[\lim_{T'\rightarrow 0}X(T')\les \norm{u_0}_{H^s}. \]
On the other hand, we get from \eqref{eq:Fsest} that
\begin{eqnarray*}
X(T')^2\les \norm{u_0}_{H^s}^2+X(T')^3+X(T')^4.
\end{eqnarray*}
If $\epsilon_0$ is sufficiently small, then we can get from
\eqref{eq:smalldata}, the continuity and the standard bootstrap that
$X(T')\les \norm{u_0}_{H^s}$ and therefore we obtain
\begin{eqnarray}\label{eq:smallFs}
\norm{u}_{F^s(T)}\les \norm{u_0}_{H^s}.
\end{eqnarray}

For $\sigma\geq s$ we obtain from Proposition \ref{plinear},
Proposition \ref{pbilinear} (a) and the energy estimate Proposition
\ref{penergy} that for any $T'\in [0,T]$ we have
\begin{eqnarray}\label{eq:Fsest2}
\left \{
\begin{array}{l}
\norm{u}_{F^{\sigma}(T')}\les \norm{u}_{E^{\sigma}(T')}+\norm{\partial_x(u^2)}_{N^{\sigma}(T')};\\
\norm{\partial_x(u^2)}_{N^{\sigma}(T')}\les \norm{u}_{F^{\sigma}(T')}\norm{u}_{F^{s}(T')};\\
\norm{u}_{E^{\sigma}(T')}^2\les
\norm{\phi}_{{H}^\sigma}^2+\norm{u}_{F^{s}(T')}\norm{u}_{F^{\sigma}(T')}^2.
\end{array}
\right.
\end{eqnarray}
Then from \eqref{eq:smallFs} we get $\norm{u}_{F^s(T)}\ll 1$ and
hence
\begin{eqnarray}
\norm{u}_{F^\sigma(T)}\les \norm{u_0}_{H^\sigma},
\end{eqnarray}
which in particularly implies \eqref{eq:H2est} as desired. We
complete the proof of part (a).

We prove now Theorem \ref{thmmain} (b), following the ideas in
\cite{IKT}. Fixing $u_0\in H^s$, then we choose $\{\phi_n\}\subset
H^\infty$ such that $\lim_{n\rightarrow \infty} \phi_n=u_0$ in
$H^s$. It suffices to prove the sequence $S_T^\infty(\phi_n)$ is a
Cauchy sequence in $C([-T,T]:H^s)$. From the definition it suffices
to prove that for any $\delta>0$ there is $M_\delta$ such that
\[\sup_{t\in [-T,T]}\norm{S_T^\infty(\phi_m)-S_T^\infty(\phi_n)}_{H^s}\leq \delta, \quad \forall m,n\geq M_\delta.\]
For $K\in \Z_+$ let $\phi_n^K=P_{\leq K}\phi_n$. Since
$\phi_n^K\rightarrow u_0^K$ in $H^2$, then we see for any fixed $K$
there is $M_{\delta,K}$ such that
\[\sup_{t\in [-T,T]}\norm{S_T^\infty(\phi_m^K)-S_T^\infty(\phi_n^K)}_{H^s}\leq \delta/2, \quad \forall m,n\geq M_{\delta,K}.\]
On the other hand, we get from Proposition \ref{penergydiff} and
Lemma \ref{pFstoHs} that
\begin{eqnarray*}
\sup_{t\in[-T,T]}\norm{S_T^\infty(\phi_n)-S_T^\infty(\phi_n^K)}_{H^s}
&\les&\norm{S_T^\infty(\phi_n)-S_T^\infty(u_n^K)}_{F^s(T)}\\
&\les&\norm{\phi_n-\phi_n^K}_{H^s}+\norm{\phi_n^K}_{H^{2s}}\norm{\phi_n-\phi_n^K}_{L^2}\\
&\les& \norm{\phi-\phi_n}_{H^s}+\norm{\phi-\phi^K}_{H^s}.
\end{eqnarray*}
Thus we obtain that for any $\delta>0$ there are $K$ and $M_\delta$
such that
\[\sup_{t\in
[-T,T]}\norm{S_T^\infty(\phi_n)-S_T^\infty(\phi_n^K)}_{H^s}\leq
\delta/2, \quad \forall n\geq M_\delta.\] Therefore, we complete the
proof of part (b) of Theorem \ref{thmmain}.

\section{Energy Estimates}

In this section we prove the energy estimates, following the ideas
in \cite{IKT}. We introduce a new Littlewood-Paley decomposition
with smooth symbols. With
\[\chi_k(\xi)=\eta_0(\xi/2^k)-\eta_0(\xi/{2^{k-1}}), \quad k\in \Z,\]
Let $\widetilde{P}_k$ denote the operator on $L^2(\R)$ defined by
the Fourier multiplier $\chi_k(\xi)$. Assume that $u,v\in
C([-T,T];L^2)$ and
\begin{eqnarray}
\left \{
\begin{array}{l}
u_t+|\partial_x|^{1+\alpha}\partial_x u=v,\ (x,t)\in \R\times (-T,T);\\
u(x,0)=\phi(x).
\end{array}
\right.
\end{eqnarray}
Then we multiply by $u$ and integrate to conclude that
\begin{equation}\label{eq:L2esti}
\sup\limits_{|t_k|\leq T}\norm{u(t_k)}_{L^2}^2\leq
\norm{\phi}_{L^2}^2+\sup\limits_{|t_k|\leq
T}\aabs{\int_{\R\times[0,t_k]}u\cdot v dxdt}.
\end{equation}

\begin{lemma}\label{lem3linear}
(a) Assume $T\in (0,1]$, $k_1,k_2,k_3 \in \Z_+$ with
$\max(k_1,k_2,k_3)\geq 1$, and $u_i\in F_{k_i}(T), i=1,2,3$. Then if
$k_{min}\leq k_{max}-5$, we have
\begin{eqnarray}
\aabs{\int_{\R\times [0,T]}u_1u_2u_3 dxdt}\les 2^{-\alpha k_{max}}
\prod_{i=1}^3 \norm{u_i}_{F_{k_i}(T)}.\label{eq:tri1}
\end{eqnarray}

(b) Assume $T\in (0,1]$, $ k\in \Z_+$, $0\leq k_1\leq k-10$, $u\in
F_k(T)$, and $v\in F_{k_1}(T)$. Then
\begin{eqnarray}\label{eq:tri2}
\aabs{\int_{\R\times
[0,T]}\widetilde{P}_k(u)\widetilde{P}_k(\partial_x u \cdot
\widetilde{P}_{k_1}(v))dxdt}\les 2^{k_1-\alpha k_{max}}
\norm{v}_{F_{k_1}(T)}\sum_{|k'-k|\leq
10}\norm{\wt{P}_{k'}(u)}_{F_{k'}(T)}^2.
\end{eqnarray}
\end{lemma}

\begin{proof}
For part (a), from symmetry we may assume $k_1\leq k_2\leq k_3$. In
order for the integral to be nontrivial we must also have
$|k_2-k_3|\leq 4$. We fix extension $\wt{u}_i \in F_{k_i}$ such that
$\norm{\wt{u}_i}_{F_{k_i}}\leq 2\norm{u_i}_{F_{k_i}(T)}$, $i=1,2,3$.
Let $\gamma:\R \rightarrow [0,1]$ denote a smooth function supported
in $[-1,1]$ with the property that
\[\sum_{n\in \Z} \gamma^3(x-n)\equiv 1, \quad x\in \R.\]
The left-hand side of \eqref{eq:tri1} is dominated by
\begin{eqnarray}\label{eq:3linear1}
&&C\sum_{|n|\leq C2^{[(1-\alpha)k_3]}}\bigg|\int_{\R\times \R} \big(
\gamma(2^{[(1-\alpha)k_3]}t-n)1_{[0,T]}(t)\wt{u}_1\big)\nonumber\\
&& \qquad \times \big(\gamma(2^{[(1-\alpha)k_3]}t-n)\wt{u}_2
\big)\cdot \big(\gamma(2^{[(1-\alpha)k_3]}t-n)\wt{u}_3\big)
dxdt\bigg|
\end{eqnarray}
We observe first that
\[|A|=|\{n: \gamma(2^{[(1-\alpha)k_3]}t-n)1_{[0,T]}(t) \mbox{ nonzero and } \ne \gamma(2^{[(1-\alpha)k_3]}t-n)\}|\leq 4.\]

We assume first that $k_1\leq k_3-5$. For the summation of $n\in
A^c$ on the left-hand side of \eqref{eq:3linear1}, as was explained
in the proof of Proposition \ref{phl}, for \eqref{eq:tri1} it
suffices to prove that if $f_{k_i,j_i}$ are $L^2$ functions
supported in ${D}_{k_i,\leq j_i}$ for $i=1,2,3$ then
\begin{eqnarray}\label{eq:3linearAc}
2^{(1-\alpha)k_3}\sum_{j_1,j_2,j_3\geq
[(1-\alpha)k_3]}|J(f_{k_1,j_1},f_{k_2,j_2},f_{k_3,j_3})|\les
2^{-\alpha k_{max}}\sum_{j_i\geq 0} \prod_{i=1}^3
2^{j_i/2}\norm{f_{k_i,j_i}}_{2}.
\end{eqnarray}
Clearly we may assume $\max(k_1,k_2,k_3)\geq 10$, otherwise we can
get \eqref{eq:3linearAc} by using Lemma \ref{lemsymes} (a).  We get
from Lemma \ref{lemsymes} (b) that the left-hand side of
\eqref{eq:3linearAc} is bounded by
\begin{eqnarray}\label{eq:3linear3}
&&2^{(1-\alpha)k_3/2}\sum_{j_1,j_2,j_3\geq
[(1-\alpha)k_3]}2^{(j_{1}+j_{2}+j_3)/2}2^{-(1+\alpha)k_3/2}\prod_{i=1}^3\norm{f_{k_i,j_i}}_{2}\nonumber\\
&&\les 2^{-\alpha k_{max}} \sum_{j_i\geq 0} \prod_{i=1}^3
2^{j_i/2}\norm{f_{k_i,j_i}}_{2},
\end{eqnarray}
which is \eqref{eq:3linearAc} as desired.

For the summation of $n\in A$, we observe that if $I\subset \R$ is
an interval, $k\in \Z_+$, $f_k\in X_k$, and $f_k^I=\ft(1_I(t)\cdot
\ft^{-1}(f_k))$ then
\begin{eqnarray}\label{eq:Xkrest}
\sup_{j\in \Z_+}2^{j/2}\norm{\eta_j(\tau-\omega(\xi))\cdot
f_k^I}_{L^2}\les \norm{f_k}_{X_k}.
\end{eqnarray}
Indeed, to prove \eqref{eq:Xkrest} it suffices to prove for any
$j_1\geq 0$ and
$f_{k,j_1}=f_k(\xi,\tau)\eta_{j_1}(\tau-\omega(\xi))$ then
\begin{eqnarray}\label{eq:Xkrest2}
\sup_{j\in \Z_+}2^{j/2}\norm{\eta_j(\tau-\omega(\xi))\cdot
f_{k,j_1}^I}_{L^2}\les 2^{j_1/2}\norm{f_{k,j_1}}_{L_2}.
\end{eqnarray}
If $j\leq j_1+20$, then \eqref{eq:Xkrest2} follows from Plancherel's
equality. If $j\geq j_1+20$ then from
\[2^{j/2}\eta_j(\tau-\omega(\xi))|f_{k,j_1}^I(\xi,\tau)|\les  2^{j/2}\eta_j(\tau-\omega(\xi))\int |f_{k,j_1}(\xi,\tau)||\tau-\tau'|^{-1}d\tau'\]
we get \eqref{eq:Xkrest2} from \eqref{eq:pXk2} since
$|\tau-\tau'|\sim 2^j$. For the summation of $n\in A$ on the
left-hand side of \eqref{eq:3linear1}, clearly we may assume
$j_1\leq 10k_3$. Then as before we can get \eqref{eq:tri1} due to
$\alpha<1$.

For part(b), we denote the commutator of $T_1,T_2$ by
$[T_1,T_2]=T_1T_2-T_2T_1$. Then the left-hand side of
\eqref{eq:tri2} is dominated by
\begin{eqnarray}\label{eq:trib1}
\aabs{\int_{\R\times [0,T]} \wt{P}_k(u) \wt{P}_k(\partial_x
u)\wt{P}_{k_1}(v)dxdt}+\aabs{\int_{\R\times [0,T]} \wt{P}_k(u)
[\wt{P}_k,\wt{P}_{k_1}(v)](\partial_x u)dxdt}.
\end{eqnarray}
For the first term in \eqref{eq:trib1} we integrate by part and then
use \eqref{eq:tri1}. For the second term it follows from
\eqref{eq:tri1} and the similar argument in the proof of Lemma 6.1
in \cite{IKT}. We omit the details.
\end{proof}

\begin{proposition}\label{penergy}
Assume that $T\in (0,1]$ and $u\in C([-T,T]:H^\infty)$ is a solution
to Eq. \eqref{eq:dgBO} on $\R\times(-T,T)$. Then for $s\geq
1-\alpha$ we have
\begin{eqnarray}\label{eq:energyeq0}
\norm{u}_{E^s(T)}^2\les
\norm{u_0}_{H^s}^2+\norm{u}_{F^{1-\alpha}(T)}\norm{u}_{F^{s}(T)}^2.
\end{eqnarray}
\end{proposition}

\begin{proof}
From definition we have
\begin{eqnarray}
\norm{u}_{E^s(T)}^2-\norm{P_{\leq 0}(u_0)}_{L^2}^2\les \sum_{k\geq
1}\sup_{t_k\in [-T,T]}2^{2sk}\norm{\wt{P}_k(u(t_k))}_{L^2}^2.
\end{eqnarray}
Then we can get from \eqref{eq:L2esti} that
\begin{eqnarray}\label{eq:energyeq}
2^{2sk}\norm{\wt{P}_k(u(t_k))}_{L^2}^2-2^{2sk}\norm{\wt{P}_k(u_0)}_{L^2}^2\les
2^{2sk}\left|\int_{\R\times [0,t_k]}\wt{P}_k(u)\wt{P}_k(u\cdot
\partial_x u)dxdt\right|.
\end{eqnarray}
It is easy to see that the right-hand side of \eqref{eq:energyeq} is
dominated by
\begin{eqnarray}\label{eq:trigoal}
&&C2^{2sk}\sum_{k_1\leq k-10}\left|\int_{\R\times
[0,t_k]}\wt{P}_k(u)\wt{P}_k(\wt{P}_{k_1}u\cdot\partial_x u)dxdt\right|\nonumber\\
&&+C2^{2sk}\sum_{k_1\geq k-9,k_2\in \Z_+}\left|\int_{\R\times
[0,t_k]}\wt{P}_k^2(u)\wt{P}_{k_1}(u)\cdot\partial_x \wt{P}_{k_2}(
u)dxdt\right|.
\end{eqnarray}
For the first term in \eqref{eq:trigoal},  using \eqref{eq:tri2}
then we get that it is bounded by
\begin{eqnarray*}
&&C2^{2sk}\sum_{k_1\leq k-10}2^{k_1-\alpha k}
\norm{u}_{F_{k_1}(T)}\sum_{|k'-k|\leq
10}\norm{\wt{P}_{k'}(u)}_{F_{k'}(T)}^2\\
&&\les \norm{u}_{F^{1-\alpha}(T)}2^{2sk}\sum_{|k'-k|\leq
10}\norm{u}_{F_{k'}(T)}^2
\end{eqnarray*}
which implies that the summation of the first term is bounded by
$\norm{u}_{F^{1-\alpha}(T)}\norm{u}_{F^{s}(T)}^2$ as desired.

For the second term in \eqref{eq:trigoal}, using \eqref{eq:tri1} we
get that it is bounded by
\begin{eqnarray*}
&&C2^{2sk}\sum_{|k_1-k|\leq 10,k_2\leq k+10}2^{k_2-\alpha
k}\norm{\wt{P}_k(u)}_{F_k(T)}\norm{\wt{P}_{k_1}(u)}_{F_{k_1}(T)}\norm{\wt{P}_{k_2}(u)}_{F_{k_2}(T)}\\
&&+C2^{2sk}\sum_{|k_1-k_2|\leq 10,k_1\geq k+10}2^{k_2-\alpha
k_2}\norm{\wt{P}_k(u)}_{F_k(T)}\norm{\wt{P}_{k_1}(u)}_{F_{k_1}(T)}\norm{\wt{P}_{k_2}(u)}_{F_{k_2}(T)}\\
&&\les 2^{2sk}\norm{u}_{F^{1-\alpha}(T)}\sum_{|k'-k|\leq
10}\norm{u}_{F_{k'}(T)}^2.
\end{eqnarray*}
Therefore, we complete the proof of the proposition.
\end{proof}

\begin{proposition}\label{penergydiff}

Let $0\leq \alpha <1$. Assume $\sigma>1-\alpha$. Let $u_1,u_2 \in
F^{\sigma}(1)$ be solutions to \eqref{eq:dgBO} with initial data
$\phi_1,\phi_2 \in H^\infty$ satisfying
\[\norm{\phi_1}_{H^{\sigma}}+\norm{\phi_2}_{H^{\sigma}}\leq \epsilon_0\ll 1.\]
Then we have
\begin{eqnarray}\label{eq:L2conti}
\norm{u_1-u_2}_{F^0(1)}\les \norm{\phi_1-\phi_2}_{L^2},
\end{eqnarray}
and
\begin{eqnarray}\label{eq:Hsconti}
\norm{u_1-u_2}_{F^{\sigma}(1)}\les
\norm{\phi_1-\phi_2}_{H^{\sigma}}+\norm{\phi_1}_{H^{2\sigma}}\norm{\phi_1-\phi_2}_{L^2}.
\end{eqnarray}
\end{proposition}

\begin{proof}
We prove first \eqref{eq:L2conti}. Since
$\norm{\phi_1}_{H^{\sigma}}+\norm{\phi_2}_{H^{\sigma}}\leq
\epsilon_0\ll 1$, then from the proof of Theorem \ref{thmmain} (a)
in the last section we know
\begin{eqnarray}\label{eq:Fssmall}
\norm{u_1}_{F^\sigma(1)}\ll 1,\quad \norm{u_2}_{F^\sigma(1)}\ll 1.
\end{eqnarray}
 Let $v=u_2-u_1$, then $v$ solves the equation
\begin{eqnarray}\label{eq:dgBOdiff}
\left\{
\begin{array}{l}
\partial_t v+|\partial_x|^{1+\alpha}\partial_x v=-\partial_x[v(u_1+u_2)/2];\\
v(0)=\phi=\phi_2-\phi_1.
\end{array}
\right.
\end{eqnarray}
Then from Proposition \ref{plinear} and Proposition \ref{pbilinear}
(b) we obtain
\begin{eqnarray}\label{eq:dgBOdiffes1}
\left\{
\begin{array}{l}
\norm{v}_{F^0(1)}\les \norm{v}_{E^0(1)}+\norm{\partial_x[v(u_1+u_2)/2]}_{N^0(1)};\\
\norm{\partial_x[v(u_1+u_2)/2]}_{N^0(1)}\les
\norm{v}_{F^0(1)}(\norm{u_1}_{F^{1-\alpha}(1)}+\norm{u_2}_{F^{1-\alpha}(1)}).
\end{array}
\right.
\end{eqnarray}

We now devote to derive an estimate on $\norm{v}_{E^0(1)}$. As in
the proof of Proposition \ref{penergy}, we get from
\eqref{eq:L2esti} that
\begin{eqnarray}\label{eq:energydiff1}
\norm{v}_{E^0(1)}^2-\norm{\phi}_{L^2}^2&\les&
\sum_{k\geq1}\left|\int_{\R\times
[0,t_k]}\wt{P}_k(v)\wt{P}_k(\partial_x £¨v£©\cdot
(u_1+u_2))dxdt\right|\nonumber\\
&& +\sum_{k\geq1}\left|\int_{\R\times
[0,t_k]}\wt{P}_k(v)\wt{P}_k(v\cdot
\partial_x (u_1+u_2))dxdt\right|.
\end{eqnarray}
For the first term on right-hand side of \eqref{eq:energydiff1},
using Lemma \ref{lem3linear} we can bound it by
\begin{eqnarray*}
&&C\sum_{k\geq 1}\sum_{k_1\leq k-10}\left|\int_{\R\times
[0,t_k]}\wt{P}_k(v)\wt{P}_k(\partial_x v\cdot\wt{P}_{k_1}(u_1+u_2))dxdt\right|\\
&&+C\sum_{k\geq 1}\sum_{k_1\geq k-9,k_2\in \Z_+}\left|\int_{\R\times
[0,t_k]}\wt{P}_k^2(v)\partial_x\wt{P}_{k_2}(v)\cdot \wt{P}_{k_1}(
u_1+u_2)dxdt\right|\\
&&\les\norm{v}_{F^0(1)}^2(\norm{u_1}_{F^{\sigma}(1)}+\norm{u_2}_{F^{\sigma}(1)}),
\end{eqnarray*}
The second term on right-hand side of \eqref{eq:energydiff1} is
dominated by
\begin{eqnarray*}
&&\sum_{k\geq1}\sum_{k_1,k_2\in \Z_+}\left|\int_{\R\times
[0,t_k]}\wt{P}_k^2(v)\wt{P}_{k_1}(v)\cdot
\partial_x \wt{P}_{k_2}(u_1+u_2)dxdt\right|\\
&&\les
\norm{v}_{F^0(1)}^2(\norm{u_1}_{F^{\sigma}(1)}+\norm{u_2}_{F^{\sigma}(1)}).
\end{eqnarray*}
Therefore, we obtain the following estimate
\begin{eqnarray}
\norm{v}_{E^0(1)}^2\les
\norm{\phi}_{L^2}^2+\norm{v}_{F^0(1)}^2(\norm{u_1}_{F^{\sigma}(1)}+\norm{u_2}_{F^{\sigma}(1)}),
\end{eqnarray}
which combined with \eqref{eq:dgBOdiffes1} implies
\eqref{eq:L2conti} in view of \eqref{eq:Fssmall}.

We prove now \eqref{eq:Hsconti}. From Proposition \ref{plinear} and
\ref{pbilinear} we obtain
\begin{eqnarray}\label{eq:dgBOdiffes2}
\left\{
\begin{array}{l}
\norm{v}_{F^{\sigma}(1)}\les \norm{v}_{E^{\sigma}(1)}+\norm{\partial_x[v(u_1+u_2)/2]}_{N^{\sigma}(1)};\\
\norm{\partial_x[v(u_1+u_2)/2]}_{N^{\sigma}(1)}\les
\norm{v}_{F^{\sigma}(1)}(\norm{u_1}_{F^{\sigma}(1)}+\norm{u_2}_{F^{\sigma}(1)}).
\end{array}
\right.
\end{eqnarray}
Since $\norm{P_{\leq 0}(v)}_{E^{\sigma}(1)}=\norm{P_{\leq
0}(\phi)}_{L^2}$, it follows from \eqref{eq:Fssmall} that
\begin{eqnarray}\label{eq:energyv}
\norm{v}_{F^{\sigma}(1)}\les \norm{P_{\geq
1}(v)}_{E^{\sigma}(1)}+\norm{\phi}_{H^{\sigma}}.
\end{eqnarray}
To bound $\norm{P_{\geq 1}(v)}_{E^{\sigma}(1)}$, we observe that
\[\norm{P_{\geq 1}(v)}_{E^{\sigma}(1)}=\norm{P_{\geq
1}(\Lambda^{\sigma}v)}_{E^0(1)},\] where $\Lambda^{\sigma}$ is the
Fourier multiplier operator with the symbol $|\xi|^{\sigma}$. Thus
we apply the operator $\Lambda^{\sigma}$ on both side of the
equation \eqref{eq:dgBOdiff} and get
\[\partial_t \Lambda^{\sigma}v+|\partial_x|^{1+\alpha}\partial_x \Lambda^{\sigma}v=-\Lambda^{\sigma}\partial_x[v(u_1+u_2)/2].\]
We rewrite the nonlinearity in the following way
\begin{eqnarray}\label{eq:dgBOdiffnonlinear}
\Lambda^{\sigma}\partial_x[v(u_1+u_2)/2]&=&\Lambda^{\sigma}[\partial_x
v(u_1+u_2)/2+v\partial_x(u_1+u_2)/2]\nonumber\\
&=&\rev{2}[\Lambda^{\sigma},(u_1+u_2)]\partial_x
v+\rev{2}[\Lambda^{\sigma},v]\partial_x(u_1+u_2)\nonumber\\
&&+\rev{2}(u_1+u_2)\Lambda^{\sigma}\partial_x
v+\rev{2}v\Lambda^{\sigma}\partial_x (u_1+u_2).
\end{eqnarray}
The right-hand side of \eqref{eq:dgBOdiffnonlinear} can be rewritten
as
\begin{eqnarray*}
\rev{2}[\Lambda^{\sigma},(u_1+u_2)]\partial_x
v+\rev{2}[\Lambda^{\sigma},v]\partial_x(u_1+u_2)+u_2\Lambda^{\sigma}\partial_x
v+v\Lambda^{\sigma}\partial_x u_1.
\end{eqnarray*}
We write the equation for $U={P}_{\geq -10}(\Lambda^{\sigma} v)$ in
the form
\begin{eqnarray}\label{eq:dgBOdiffU}
\left\{
\begin{array}{l}
\partial_t U+|\partial_x|^{1+\alpha}\partial_x U=P_{\geq -10}(-u_2\cdot\partial_x U)+P_{\geq -10}(G);\\
U(0)=P_{\geq -10}(\Lambda^{\sigma}\phi),
\end{array}
\right.
\end{eqnarray}
where
\begin{eqnarray*}
G&=&-P_{\geq -10}(u_2)\cdot \Lambda^{\sigma}\partial_x P_{\leq
-11}(v)-P_{\leq -11}(u_2)\cdot \Lambda^{\sigma}\partial_x P_{\leq
-11}(v)\\
&&-\rev{2}[\Lambda^{\sigma},(u_1+u_2)]\partial_x
v-\rev{2}[\Lambda^{\sigma},v]\partial_x(u_1+u_2)-v\cdot
\Lambda^{\sigma}\partial_x u_1.
\end{eqnarray*}

It follows from \eqref{eq:L2esti} and \eqref{eq:dgBOdiffU} that
\begin{eqnarray*}
\norm{U}_{E^0(1)}^2-\norm{\phi}_{H^{\sigma}}^2&\les& \sum_{k\geq 1}
\left|\int_{\R\times [0,t_k]}\wt{P}_k(U)\wt{P}_{k}(u_2\cdot
\partial_x U)dxdt\right|\\
&& +\sum_{k\geq 1} \left|\int_{\R\times [0,t_k]}\wt{P}_k^2(U)P_{\geq
-10}(u_2)\cdot \Lambda^{\sigma}\partial_x P_{\leq
-11}(v)dxdt\right|\\
&&+\sum_{k\geq 1} \left|\int_{\R\times
[0,t_k]}\wt{P}_k(U)[\Lambda^{\sigma},(u_1+u_2)]\partial_x
vdxdt\right|\\
&&+\sum_{k\geq 1} \left|\int_{\R\times
[0,t_k]}\wt{P}_k(U)[\Lambda^{\sigma},v]\partial_x(u_1+u_2)dxdt\right|\\
&&+\sum_{k\geq 1} \left|\int_{\R\times [0,t_k]}\wt{P}_k(U)v\cdot
\Lambda^{\sigma}\partial_x u_1dxdt\right|\\
&:=&I+II+III+IV+V.
\end{eqnarray*}
For the contribution of $I$ we can bound it as in
\eqref{eq:energydiff1} and then get that
\[I\les \norm{U}_{F^0(1)}^2\norm{u_2}_{F^{\sigma}(1)}. \]
For the contribution of $II$, since the derivatives fall on the low
frequency, then we can easily get
\[II\les \norm{U}_{F^0(1)}^2\norm{u_2}_{F^{\sigma}(1)}. \]
We consider now the contribution of $V$.
\begin{eqnarray*}
V&\les& \sum_{k\geq 1} \sum_{k_1,k_2\in \Z_+}\left|\int_{\R\times
[0,t_k]}\wt{P}_k(U)\cdot \wt{P}_{k_1}(v)\cdot
\Lambda^{\sigma}\partial_x \wt{P}_{k_2}(u_1)dxdt\right|\\
&\les& \sum_{k\geq 1}\sum_{|k-k_2|\leq 5, k_1 \leq k-10}
2^{k(\sigma+1-\alpha)}
\norm{\wt{P}_{k}(U)}_{F_k(1)}\norm{\wt{P}_{k_1}(v)}_{F_{k_1}(1)}\norm{\wt{P}_{k_2}(u_1)}_{F_{k_2}(1)}\\
&&+\sum_{k\geq 1}\sum_{k_1 \geq k-10}
2^{k_2(\sigma+1)}2^{-\alpha\max(k_1,k_2)}
\norm{\wt{P}_{k}(U)}_{F_k(1)}\norm{\wt{P}_{k_1}(v)}_{F_{k_1}(1)}\norm{\wt{P}_{k_2}(u_1)}_{F_{k_2}(1)}\\
&\les&
\norm{U}_{F^0(1)}\norm{v}_{F^0(1)}\norm{u_1}_{F^{2\sigma}(1)}+\norm{U}_{F^0(1)}^2\norm{u_1}_{F^{\sigma}(1)}.
\end{eqnarray*}
For the contribution of $III$, we obtain
\begin{eqnarray*}
III&\les&\sum_{k\geq 1}\sum_{k_1\leq k_2-10} \left|\int_{\R\times
[0,t_k]}\wt{P}_k(U)[\Lambda^{\sigma},\wt{P}_{k_1}(u_1+u_2)]\partial_x
\wt{P}_{k_2}(v)dxdt\right|\\
&&+ \sum_{k\geq 1}\sum_{k_1\geq k_2-9}\left|\int_{\R\times
[0,t_k]}\wt{P}_k(U)[\Lambda^{\sigma},\wt{P}_{k_1}(u_1+u_2)]\partial_x
\wt{P}_{k_2}(v)dxdt\right|\\
&:=&III_1+III_2.
\end{eqnarray*}
We note that in the term $III_2$, the component $(u_1+u_2)$ can
spare derivative, and thus we get
\begin{eqnarray*}
III_2&\les&\sum_{k\geq 1}\sum_{k_1\geq k_2-9}\left|\int_{\R\times
[0,t_k]}\Lambda^{\sigma}\wt{P}_k(U)\wt{P}_{k_1}(u_1+u_2)\partial_x
\wt{P}_{k_2}(v)dxdt\right|\\
&&+\sum_{k\geq 1}\sum_{k_1\geq k_2-9}\left|\int_{\R\times
[0,t_k]}\wt{P}_k(U)\wt{P}_{k_1}(u_1+u_2)\Lambda^{\sigma}\partial_x
\wt{P}_{k_2}(v)dxdt\right|\\
&\les&\sum_{k\geq 1}\sum_{k_1\geq k_2-9} 2^{-\alpha
k_1}2^{k\sigma}2^{k_2}\norm{\wt{P}_{k}(U)}_{F_{k}(1)}\norm{\wt{P}_{k_1}(u_1+u_2)}_{F_{k_1}(1)}\norm{\wt{P}_{k_2}(v)}_{F_{k_2}(1)}\\
&&+\sum_{k\geq 1}\sum_{k_1\geq k_2-9} 2^{-\alpha
k_1}2^{k_2(\sigma+1)}\norm{\wt{P}_{k}(U)}_{F_{k}(1)}\norm{\wt{P}_{k_1}(u_1+u_2)}_{F_{k_1}(1)}\norm{\wt{P}_{k_2}(v)}_{F_{k_2}(1)}\\
&\les&\norm{U}_{F^0(1)}^2(\norm{u_1}_{F^\sigma(1)}+\norm{u_2}_{F^\sigma(1)})
\end{eqnarray*}
For the contribution of $III_1$ we need to exploit the cancelation
of the commutator. By taking $\gamma$ and extending $U,u_1,u_2,v$ as
in the proof of Lemma \ref{lem3linear}, then we get
\begin{eqnarray*}
III_1&\les&\sum_{k\geq 1}\sum_{k_1\leq k_2-10}\sum_{|n|\leq
C2^{[(1-\alpha)k_3]}} \int_{\R\times \R}\big(
\gamma(2^{[(1-\alpha)k_3]}t-n)1_{[0,t_k]}(t)\wt{P}_k(U)\big)\\
&&\cdot
[\Lambda^{\sigma},\gamma(2^{[(1-\alpha)k_3]}t-n)\wt{P}_{k_1}(u_1+u_2)]\partial_x
\wt{P}_{k_2}(\gamma(2^{[(1-\alpha)k_3]}t-n)v)dxdt
\end{eqnarray*}
Let $f_k=\gamma(2^{[(1-\alpha)k_3]}t-n)\wt{P}_k(U)$,
$g_{k_1}=\gamma(2^{[(1-\alpha)k_3]}t-n)\wt{P}_{k_1}(u_1+u_2)$ and
$h_{k_2}=\wt{P}_{k_2}(\gamma(2^{[(1-\alpha)k_3]}t-n)v)$. It is easy
to see from $|k_2-k|\leq 3$ that
\[
|\ft([\Lambda^{\sigma},g_{k_1}]\partial_x h_{k_2})(\xi,\tau)|\les
\int_{\R\times \R}
|\widehat{g}_{k_1}(\xi-\xi',\tau-\tau')|2^{k_1}2^{\sigma
k_2}|\widehat{h}_{k_2}(\xi',\tau')|d\xi'd\tau'.
\]
Then using a similar argument in the proof of Lemma \ref{lem3linear}
we can get that
\begin{eqnarray*}
III_1&\les&\sum_{k\geq 1}\sum_{k_1\leq k_2-10} 2^{k_1}2^{\sigma k_2}2^{-\alpha k_2}\norm{\wt{P}_{k}(U)}_{F_{k}(1)}\norm{\wt{P}_{k_1}(u_1+u_2)}_{F_{k_1}(1)}\norm{\wt{P}_{k_2}(v)}_{F_{k_2}(1)}\\
&\les&\norm{U}_{F^0(1)}^2(\norm{u_1}_{F^\sigma(1)}+\norm{u_2}_{F^\sigma(1)}).
\end{eqnarray*}
The contribution of $IV$ is identical to the one of $III$ from
symmetry. Therefore, we have proved that
\begin{eqnarray*}
\norm{U}_{E^0(1)}^2&\les
&\norm{\phi}_{H^\sigma}^2+\norm{U}_{F^0(1)}^2(\norm{u_1}_{F^\sigma(1)}+\norm{u_2}_{F^\sigma(1)})\\
&&+\norm{U}_{F^0(1)}\norm{v}_{F^0(1)}\norm{u_1}_{F^{2\sigma}(1)}.
\end{eqnarray*}
By \eqref{eq:Fssmall}, Theorem \ref{thmmain} (a), \eqref{eq:L2conti}
and \eqref{eq:energyv} we get
\[\norm{U}_{E^0(1)}\les \norm{\phi_1-\phi_2}_{H^\sigma}+\norm{\phi_1-\phi_2}_{L^2}\norm{\phi_1}_{H^{2\sigma}},\]
which combined with \eqref{eq:energyv} completes the proof of the
proposition.
\end{proof}

\noindent{\bf Acknowledgment.} The author would like to thank
Professor Carlos E. Kenig for helpful suggestions. This work is
supported in part by RFDP of China No. 20060001010, the National
Science Foundation of China, grant 10571004; and the 973 Project
Foundation of China, grant 2006CB805902, and the Innovation Group
Foundation of NSFC, grant 10621061.

\end{document}